\documentclass[preprint,12pt]{article}
\usepackage[top=1.1in, bottom=1.2in, left=0.8in, right=0.8in]{geometry}
\usepackage{amsmath,amsfonts,amssymb,amsthm}
\usepackage{mathrsfs,dsfont}
\usepackage{graphicx}
\usepackage{colortbl,dcolumn}
\usepackage{psfrag}
\usepackage{booktabs}
\usepackage{cite}
\usepackage{subfigure}
\usepackage{hyperref}
\usepackage{listings}
\usepackage{algorithm, algorithmic}
\usepackage{comment}
\allowdisplaybreaks[4]
\numberwithin{equation}{section}

\newcommand{\R}{\mathbb{R}}
\newcommand{\N}{\mathbb{N}}
\newcommand{\E}{\mathbb{E}}
\renewcommand{\P}{\mathbb{P}}
\newcommand{\tr}{\operatorname{Tr}}

\newcommand{\diff}{{\,\rm{d}}}
\newcommand{\dif}{{\rm{d}}}

\newcommand{\F}{\mathcal{F}}

\newtheorem{theorem}{Theorem}[section]

\newtheorem{lemma}[theorem]{Lemma}
\newtheorem{proposition}[theorem]{Proposition}

\newtheorem{example}[theorem]{Example}
\newtheorem{assumption}[theorem]{Assumption}

\begin{document}
	
\title{Weak convergence order of stochastic theta method for SDEs driven by time-changed L\'{e}vy noise\footnotemark[1]}
	

\footnotetext{\footnotemark[1] This work was supported by National Natural Science Foundation of China (Nos. 12501581, 12201552), Yunnan Fundamental Research Projects (No. 202301AU070010), and the Disciplinary Funding of Central University of Finance and Economics.}

	
\author{Ziheng Chen\footnotemark[2],
	    \quad Jiao Liu\footnotemark[3],
		\quad Meng Cai\footnotemark[4]}
	
\footnotetext{\footnotemark[2] School of Mathematics and Statistics, Yunnan University, Kunming, Yunnan, 650500, China. Email: czh@ynu.edu.cn.}
	
\footnotetext{\footnotemark[3] School of Mathematics and Statistics, Yunnan University, Kunming, Yunnan, 650500, China. Email: liujiao3@stu.ynu.edu.cn. }
	
\footnotetext{\footnotemark[4] School of  Statistics and Mathematics, Central University of Finance and Economics, Beijing, 100081, China. Email: mcai@lsec.cc.ac.cn. Corresponding author.}
	
\date{}
\maketitle

\begin{abstract}
      {\rm\small
      This paper studies the weak convergence order of the stochastic theta method for stochastic differential equations (SDEs) driven by time-changed L\'{e}vy noise under global Lipschitz and linear growth conditions. In contrast to classical L\'{e}vy-driven SDEs, the presence of a random time change makes the weak error analysis involve both the discretization error of the underlying equation and the approximation error of the random clock. Moreover, compared with explicit Euler--Maruyama method, the implicit drift correction in the stochastic theta method makes the associated weak error analysis substantially more delicate. To address these difficulties, we first establish a global weak convergence estimate of order one for the stochastic theta method applied to the corresponding non-time-changed L\'{e}vy SDEs on the infinite time interval by means of the Kolmogorov backward partial integro differential equations. Incorporating the approximation of the inverse subordinator together with the duality principle, we derive the weak convergence order of the stochastic theta method with $\theta \in [0,1]$ in the time-changed L\'{e}vy setting. The result advances the currently available weak convergence analysis beyond the Euler–Maruyama method to the more general class of stochastic theta method, and establishes a workable route from the weak analysis of the underlying non-time-changed L\'{e}vy equation to the corresponding time-changed problem. Finally, numerical experiments are presented to further support the theoretical findings.
      } \\
		
      \textbf{AMS subject classification: }
      {\rm\small 60H10, 60H35, 65C30}\\
      
      \textbf{Key Words: }{\rm\small Time-changed L\'{e}vy SDEs; Stochastic theta method; Weak convergence order; Kolmogorov backward partial integro differential equations}
\end{abstract}

\section{Introduction}
Complex dynamical systems, influenced simultaneously by random waiting times, subdiffusion, burst-like jumps, and non-Gaussian perturbations, arise naturally in finance, physics, biology, and engineering \cite{ni2020study}. In contrast to classical SDEs evolving under a deterministic clock, time-changed stochastic models incorporate a random time transformation and are therefore better suited to describing trapping effects, random activity periods, and anomalous diffusion; see, e.g., \cite{magdziarz2009stochastic, nane2021timechanged}. Motivated by this, we consider the following SDEs driven by time-changed L\'{e}vy noise
\begin{align}\label{eq:TCSDE}
	  \diff{X(t)}
	  =&~
	  f(E(t),X(t-))\diff{E(t)}
	  +
	  g(E(t),X(t-))\diff{W(E(t))} \notag
	  \\&~+ 
     \int_{|z|<c} h(E(t),X(t-),z)  
	  \widetilde{N}(\dif{z},\dif{E(t)}), 
	  \quad t \in [0,T]
\end{align}
with $X(0) = x_{0} \in \R^{d}$, where $W$ is an $m$-dimensional standard Brownian motion, and $\widetilde{N}$ is the compensator Poisson random measure. Here, the  random time-change process $\{E(t)\}_{t \geq 0}$ is used to describe the discrepancy between physical time and operational time, thereby providing a natural way to describe memory effects, trapping phenomena, and anomalous diffusion. Combined with L\'{e}vy jump perturbations, this random time-change structure yields a flexible framework for dynamical systems with random clocks and sudden transitions. For this reason, such equations have attracted considerable attention in recent years, and substantial progress has been made in their study, including existence, uniqueness, moment estimates, stability analysis, and their connections with fractional Fokker--Planck equations; see, e.g., \cite{magdziarz2009stochastic, nane2016stochastic, nane2017stability, nane2018path, nane2021timechanged, ni2020study} and references therein.

Since exact solutions to time-changed SDEs are rarely available in explicit form, numerical methods play an important role in investigating their dynamical behavior. For SDEs driven by time-changed Brownian motion (i.e., $h(\cdot,\cdot,\cdot) \equiv 0$ in \eqref{eq:TCSDE}), the seminal work \cite{jum2016strong} establishes the strong and weak convergence orders of Euler--Maruyama method under the global Lipschitz and linear growth conditions. Subsequently, a lot of works investigated the strong convergence rates of numerical methods under less and less restrictive conditions \cite{deng2020semiimplicit, li2021truncated, liu2020truncated, liu2023milstein}. For SDEs driven by time-changed L\'{e}vy noise, \cite{chen2026strong} establishes the strong convergence order of stochastic theta method under the global Lipschitz and linear growth conditions. Such results primarily describe how well numerical solutions approximate exact solutions in a pathwise sense. In many applications, however, one is more concerned with the effective approximation of expected functionals and other distribution-related quantities, for which weak convergence analysis provides a natural measure of accuracy \cite{milstein2021stochastic, platen2010numerical}. For SDEs driven by time-changed L\'{e}vy noise, the weak convergence theory is currently available only for the Euler--Maruyama method \cite{chen2026strong}. This naturally raises the question of whether similar weak convergence results remain valid for more general numerical methods. Among them, the stochastic theta method, as a parameterized family unifying explicit and implicit method, offers distinct advantages for problems involving stiff drift terms or enhanced stability requirements. However, establishing its weak convergence order is mathematically challenging because of the intricate interplay between the implicit correction and the jump perturbations. In this paper, we address this problem by establishing the weak convergence order of the stochastic theta method for SDEs driven by time-changed L\'{e}vy noise.

To carry out the weak convergence analysis, we consider the corresponding non-time-changed L\'{e}vy SDEs on the infinite time interval $[0,+\infty)$,
\begin{equation}\label{eq:SDE}
\left\{\begin{aligned}
      \diff{Y(t)}
      =&~
	  f(t,Y(t-))\diff{t}
	  +  
	  g(t, Y(t-))\diff{W(t)}
	  +
	  \int_{|z|<c} h(t,Y(t-),z)
      \,\widetilde{N}(\dif{z},\dif{t}),
	  \quad t \geq 0,
      \\
      Y(0) =&~ x_{0},
\end{aligned}\right.
\end{equation} 
 and its stochastic theta discretization for $\theta \in [0,1]$,
\begin{align}\label{eq:Stmethod}
	  Y_{n+1}
	  =&~
	  Y_{n} + \theta f(t_{n+1},Y_{n+1})\Delta
	  +
	  (1-\theta) f(t_{n},Y_{n})\Delta
      +
      g(t_{n},Y_{n})\Delta W_{n}  \notag
      \\&~+
      \int_{t_{n}}^{t_{n+1}}\int_{|z|<c}
      h(t_n,Y_n,z) \,\widetilde{N}(\dif{z},\dif{s}),
      \quad n = 0,1,2,\cdots
\end{align}
with $Y_0=Y(0)$, $\Delta \in (0,1)$, $t_{n} := n\Delta$ for $n \in \N$, and $\Delta W_{n} := W(t_{n+1}) - W(t_n)$. The implicit drift term $f(t_{n+1},Y_{n+1})$ makes the numerical update depend on the non-$\F_{t_n}$-adapted quantity $Y_{n+1}$, so that the usual It\^{o}-formula-based local weak expansion available for Euler--Maruyama method can no longer be derived directly for the stochastic theta method; see \cite{chen2026strong} for more details. Our strategy is therefore to separate the analysis into an explicit reference step and an implicit correction step. More precisely, we first introduce a frozen Euler approximation $\overline{Y}_{n+1}$, obtained by freezing the coefficients at $(t_n,Y_n)$, as an explicit reference step, for which a standard generator-based weak expansion can be derived as follows. The Brownian part is handled by a third-order Taylor expansion together with the fact that all odd moments of a Gaussian random variable vanish, which yields a remainder of order $\mathcal{O}(\Delta^{2})$. Meanwhile, the jump part is treated under the finite activity condition $\mu(\{|z| < c\}) < +\infty$ by conditioning on the number of jumps, where the probability of two or more jumps in one step is $\mathcal{O}(\Delta^{2})$. Consequently, 
\begin{equation}\label{eq:frozenEuler}
      \E\big[u(t_{n+1},\overline{Y}_{n+1}) 
      \mid \F_{t_n}\big] =u(t_n, Y_n) 
      + \mathcal{O}\big((1+|Y_{n}|^{q})\Delta^{2}\big),  
\end{equation}
where $u(s,x) := \E[\Phi(Y(t)) \mid Y(s) = x]$ for $0 \leq s \leq t, x \in \R^{d}$, with $\Phi \in C_{p}^{4}(\R^{d};\R)$ denotes the unique solution of the Kolmogorov backward partial integro differential equations (KB-PIDEs) \eqref{eq:Kbpide}. Here and below, $\mathcal{O}\big((1+|Y_{n}|^{q}) \Delta^{2}\big)$ denotes an $\F_{t_n}$-measurable random variable $R_n$ such that
$|R_n| \leq C(1+|Y_{n}|^{q}) \Delta^{2}$ holds almost surely for some deterministic constant $C > 0$ independent of $n$ and $\Delta$. We then compare the stochastic theta approximation $Y_{n+1}$ with the frozen Euler step $\overline{Y}_{n+1}$. By linearizing the implicit equation and using a small-step absorption argument, we show that the implicit correction $Y_{n+1}-\overline{Y}_{n+1}$ contributes only a higher-order weak term:
\begin{equation}\label{eq:implicitcorrection}
      \E\big[u(t_{n+1},Y_{n+1}) \mid \F_{t_n}\big]
      =
      \E\big[u(t_{n+1},\overline{Y}_{n+1}) 
      \mid \F_{t_n}\big]
      +
      \mathcal{O}\big((1+|Y_{n}|^{q})\Delta^{2}\big).    
\end{equation}
This step captures the essential technical difficulty introduced by the implicit drift term and provides the key control needed beyond explicit schemes. Combining \eqref{eq:frozenEuler} and \eqref{eq:implicitcorrection} yields the desired one-step weak consistency estimate for the stochastic theta method. A telescoping argument together with appropriate moment bounds then gives a global weak error of order one for the corresponding non-time-changed L\'{e}vy SDEs on the infinite time interval. Incorporating the approximation of the inverse subordinator and the duality principle, we further obtain weak convergence of order one for the original time-changed equation in Theorem \ref{th:weakorder}. In this way, the present analysis advances the currently available weak convergence theory beyond the Euler--Maruyama method to the more general class of stochastic theta method, and also provides a workable route from the underlying non-time-changed L\'{e}vy SDEs to the corresponding time-changed problem.


The remainder of this paper is organized as follows. Section \ref{sect:asspre} introduces the assumptions, notation, and preliminary results used throughout the paper. In Section \ref{sec:weakorder}, we establish the weak convergence order of the stochastic theta method for the considered time-changed L\'{e}vy SDE, with the corresponding non-time-changed L\'{e}vy equation on the infinite time interval serving as the core analytical step. Section \ref{sect:numexp} presents numerical experiments to illustrate the theoretical results. Finally, Section \ref{sec:conclusion} concludes the paper.

\section{Assumptions and preliminary results}\label{sect:asspre}
Throughout this paper, we use the following notations unless otherwise specified. Let $\R = (-\infty,+\infty)$, $\R^{+} = (0,+\infty)$, $\N = \{0,1,2,\cdots\}$ and $\N^{+} = \{1,2,\cdots\}$. For any $a,b \in \R$, we denote $a \wedge b := \min\{a, b\}$. Let $\langle x,y \rangle$ denote the inner product of vectors $x,y \in \R^d$ and $|\cdot| := \sqrt{\langle \cdot,\cdot \rangle}$ be the corresponding Euclidean norm in $\R^d$. By $A^{\top}$ we denote the transpose of a vector in $\R^{d}$ or a matrix in $\R^{d \times m}$. Following the same notation as the vector norm, $|A| := \sqrt{\tr{(A^{\top}A)}}$ denotes the trace norm of a matrix $A \in \R^{d \times m}$. For simplicity, the letter $C$ denotes a generic positive constant that is independent of $\Delta \in (0,1), t \geq 0$ and varies for each appearance.

%
%
Let $(\Omega,\F,\P,\{\F_{t}\}_{t \geq 0})$ be a complete filtered probability space  with the filtration $\{\F_{t}\}_{t \geq 0}$ satisfying the usual conditions.  Let $\{D(t)\}_{t \geq 0}$ be an $\{\F_{t}\}_{t \geq 0}$-adapted subordinator with Laplace exponent $\phi$ and L\'{e}vy measure $\nu$, i.e., $\{D(t)\}_{t \geq 0}$ is a one-dimensional non-decreasing L\'{e}vy process with c\`{a}dl\`{a}g paths starting at $0$ with the Laplace transform
\begin{equation}\label{eq:2.1}
	  \E\big[e^{-\lambda D(t)}\big]
	  =
	  e^{-t\phi(\lambda)},
	  \quad \lambda > 0, \ t \geq 0,
\end{equation}
where the Laplace exponent $\phi(\lambda) := a\lambda + \int_{0}^{+\infty}(1-e^{-\lambda x}) \,\nu(\dif{x}), \lambda > 0, a \geq 0$ with $\int_{0}^{+\infty} (x \wedge 1) \,\nu(\dif{x}) <+\infty$. We assume that L\'{e}vy measure $\nu$ is infinite, i.e., $\nu((0,+\infty)) = +\infty$, which implies that $\{D(t)\}_{t \geq 0}$ has strictly increasing paths with infinitely many jumps (see, e.g., \cite[Theorem 21.3]{sato1999levy}). Defining the inverse subordinator of $\{D(t)\}_{t \geq 0}$ by
$$E(t) := \inf\big\{s \geq 0 : D(s) > t \big\}, \quad t \geq 0,$$
it follows from $\{D(t)\}_{t \geq 0}$ having strictly increasing paths that $t \mapsto E(t)$ is continuous and non-decreasing almost surely. Besides, the inverse subordinator $\{E(t)\}_{t \geq 0}$ is a continuous $\{\F_{t}\}_{t \geq 0}$-adapted time change (see, e.g., \cite[Lemma 2.7]{kobayashi2011stochastic}), and hence the time-changed filtration $\{\F_{E(t)}\}_{t \geq 0}$ is well defined. Moreover, \cite[Lemma 2.3]{deng2020semiimplicit} indicates that for any $\lambda > 0$ and $t \geq 0$, there exists $C > 0$, depending on $\lambda$ but independent of $t$, such that 
\begin{equation}\label{eq:estimateEelambdaEt}
      \E\big[e^{\lambda E(t)}\big] \leq Ce^{Ct}.  
\end{equation}

Concerning the underlying model \eqref{eq:TCSDE}, denote $X(t-) := \lim_{s \uparrow t}X(s), t > 0$. Let $\{W(t)\}_{t \geq 0}$ be an $m$-dimensional standard Brownian motion on $(\Omega,\F,\P,\{\F_{t}\}_{t \geq 0})$. Assume that $\widetilde{N}(\dif{z},\dif{t}) := N(\dif{z},\dif{t}) - \mu(\dif{z})\dif{t}$ is the compensator of $\{\F_{t}\}_{t \geq 0}$-adapted Poisson random measure $N(\dif{z},\dif{t})$ defined on $(\Omega \times (\R^{d} \backslash \{0\}) \times \R^{+},\F \times \mathcal{B}((\R^{d} \backslash \{0\})) \times \mathcal{B}(\R^{+}))$, where $\mu$ is a L\'{e}vy measure satisfying $\int_{\R^{d} \backslash \{0\}} \big(|z|^{2} \wedge 1\big) \,\mu(\dif{z}) < +\infty$ and $\mu(\{|z| < c\}) < +\infty$ for any $c > 0$. At last, we always assume that the processes $\{D(t)\}_{t \geq 0}$, $\{W(t)\}_{t \geq 0}$ and $\{N(\cdot,t)\}_{t \geq 0}$ are mutually independent. To establish the well-posedness of \eqref{eq:TCSDE}, we first assume that the coefficients $f \colon \R^{+} \times \R^{d} \to \R^{d}, g  \colon \R^{+} \times \R^{d} \to \R^{d \times m}$ and $h \colon  \R^{+} \times \R^{d} \times \R^{d} \to \R^{d}$ satisfy the following global Lipschitz and linear growth conditions.

\begin{assumption}\label{ass:globallinear}
      There exist constants $L > 0$ and $C > 0$ such that for any $t \geq 0$ and $x,y \in \R^{d}$,
	   \begin{equation*}
			|f(t,x)-f(t,y)| \leq L|x-y|,
	   \end{equation*}
      and 
	   \begin{equation*}
			|g(t,x)-g(t,y)|^{2}
			+
			\int_{|z|<c} |h(t,x,z) - h(t,y,z)|^{2} \,\mu(\dif{z})
			\leq
			C|x-y|^{2}.
	   \end{equation*}
      Besides, there exists $C > 0$ such that for any $t \geq 0$ and $x \in \R^{d}$,
      \begin{equation*}
			|f(t,x)|^{2} + |g(t,x)|^{2}
			+
			\int_{|z|<c} |h(t,x,z)|^{2} \,\mu(\dif{z})
			\leq
			C(1+|x|^{2}).
	  \end{equation*}
\end{assumption}

Under Assumption \ref{ass:globallinear}, standard results (see, e.g., \cite[Theorem 6.2.3]{applebaum2009levy}) imply that \eqref{eq:SDE} possesses a unique c\`{a}dl\`{a}g $\{\F_t\}_{t \geq 0}$-adapted solution $\{Y(t)\}_{t \geq 0}$, given by
\begin{align}\label{eq:SDEintegral}
      Y(t)
	  =&~
	  x_{0} + \int_{0}^{t} f(s,Y(s-)) \diff{s}
	  +
      \int_{0}^{t} g(s, Y(s-)) \diff{W(s)} \notag
	  \\&~+
	  \int_{0}^{t} \int_{|z|<c} h(s,Y(s-),z)
      \,\widetilde{N}(\dif{z},\dif{s}), \quad t \geq 0. 
\end{align} 
Let $\mathcal{G}_{t} := \F_{E{(t)}}$ for $t \geq 0$ and denote $\mathcal{L}(\mathcal{G}_{t})$ the class of left continuous with right limits and 
$\{\mathcal{G}_{t}\}_{t \geq 0}$-adapted  processes.

\begin{assumption}\label{ass:techniquecond}
      If $\{X(t)\}_{t \geq 0}$ is a c\`{a}dl\`{a}g and $\{\mathcal{G}_{t}\}_{t \geq 0}$-adapted stochastic process, then
      $$f(E(t),X(t-)),g(E(t),X(t-)),h(E(t),X(t-),z)
        \in \mathcal{L}(\mathcal{G}_t),
        \quad z \in \R^{d}, t \geq 0.$$
\end{assumption}

Assumption \ref{ass:techniquecond} ensures that the time-changed integrands in \eqref{eq:TCSDE} are c\`{a}dl\`{a}g $\{\mathcal{G}_{t}\}_{t \geq 0}$-adapted  processes, so that the corresponding stochastic integrals are well-defined \cite{kobayashi2011stochastic}. Combining \eqref{eq:SDEintegral}, Assumption \ref{ass:techniquecond}, and \cite[Lemma 4.1]{kobayashi2011stochastic}, we conclude that \eqref{eq:TCSDE} admits a unique c\`{a}dl\`{a}g
$\{\mathcal{G}_{t}\}_{t \geq 0}$-adapted
solution $\{X(t)\}_{t \in [0,T]}$, described by
\begin{align}\label{eq:TCSDEintegral}
      X(t)
	   =&~
	   x_{0} + \int_{0}^{t} f(E(s),X(s-)) \diff{E(s)}
	   +
	   \int_{0}^{t} g(E(s),X(s-)) \diff{W(E(s))} \notag
	   \\&~+ 
	   \int_{0}^{t} \int_{|z|<c} 
	   h(E(s),X(s-),z) \,\widetilde{N}(\dif{z},\dif{E(s)}).
      \quad t \in [0,T],
\end{align}
For notational convenience, we write $X(t)$ and $Y(t)$ for $X(t-)$ and $Y(t-)$, respectively; see \cite[Remark 2.1]{dareiotis2016tamed}. The following duality principle (see, e.g., \cite[Theorem 4.2]{kobayashi2011stochastic} and \cite[Section 3.3]{ni2020study}) links the time-changed equation \eqref{eq:TCSDE} with the corresponding non-time-changed equations \eqref{eq:SDE}.

\begin{lemma}\label{lem:dualityprinciple}
      Suppose that Assumptions \ref{ass:globallinear} and \ref{ass:techniquecond} hold. Then we have the following conclusions:
	  \begin{itemize}
            \item [(1)] If $\{Y(t)\}_{t \geq 0}$ satisfies \eqref{eq:SDE}, then $\{Y(E(t))\}_{t \geq 0}$ satisfies \eqref{eq:TCSDE}.
				
            \item [(2)] If $\{X(t)\}_{t \geq 0}$ satisfies \eqref{eq:TCSDE}, then $\{X(D(t))\}_{t \geq 0}$ satisfies \eqref{eq:SDE}.
      \end{itemize}
\end{lemma}

To proceed, we impose a strengthened linear growth condition on the jump coefficient.

\begin{assumption}\label{As:fghgrowthcond}
      For any $p \geq 2$, there exists a constant $C > 0$ such that for all $t \geq 0$ and $x \in \R^{d}$,  
	  \begin{equation*}
            \int_{|z|<c}|h(t,x,z)|^{p} \,\mu(\dif{z})
		    \leq
		    C(1+|x|^{p}).
	  \end{equation*}
\end{assumption}

We next recall two auxiliary estimates for the process $\{Y(t)\}_{t \geq 0}$, namely a higher-order moment bound and a short-time weak increment estimate; see \cite[Lemmas 4.1 and 4.3]{chen2026strong} for the proofs. For later use, let $C_{p}^{k}(\R^{d};\R)$ denote the class of $k \in \{1,2,\cdots\}$-times continuously differentiable functions from $\R^{d}$ to $\R$ such that the function itself and all its partial derivatives up to order $k$ have at most polynomial growth.

\begin{lemma}\label{lm:Ytpexactbound}
      Suppose that Assumptions \ref{ass:globallinear} and \ref{As:fghgrowthcond} hold. Then for any $p \geq 2$, there exists a constant $C > 0$, independent of $t$, such that
      \begin{equation*}
			\E\big[|Y(t)|^{p}\big]
			\leq 
			Ce^{Ct}, \quad t \geq 0.  
      \end{equation*}
      Morteover, for any $\Phi \in C_{p}^{2}(\R^{d};\R)$ and any $t \geq s \geq 0$, there exists a constant $C > 0$, independent of $t$, such that
      \begin{align*}
			\big|\E\big[\Phi(Y(t)) - \Phi(Y(s))\big]\big| 
			\leq 
			C(t-s)e^{Ct}.
      \end{align*}
\end{lemma}


\section{Weak convergence order of stochastic theta method}\label{sec:weakorder}
This section analyzes the weak convergence order of numerical approximation for the time-changed equation \eqref{eq:TCSDE} by studying the stochastic theta method \eqref{eq:Stmethod} for the corresponding non-time-changed equation \eqref{eq:SDE}. For any $\theta \in [0,1]$, the considered method is well-posed provided that $\theta{L}\Delta \leq \frac{1}{2}$; see \cite[Theorem C.2]{stuart1996dynamical}. We begin by establishing bounded high-order moments of the numerical solution $\{Y_{n}\}_{n \in \N}$, which forms a basic ingredient of the subsequent weak error analysis.

\subsection{$p$th-moment boundedness of the stochastic theta approximation}
To establish $p$th-moment boundedness of the stochastic theta approximation $\{Y_{n}\}_{n \in \N}$, we first study the inverse of the mapping
\begin{equation}\label{eq:mapFtx}
      F_{t}(x) 
      := 
      x - \theta\Delta f(t,x),
      \quad t \geq 0, x \in \R^{d},     
\end{equation}
for which the following lemma provides the global Lipschitz continuity and linear growth estimates.

\begin{lemma}
      Suppose that Assumption \ref{ass:globallinear} holds, and let $\theta{L}\Delta \leq \frac{1}{2}$ with $\theta \in [0,1]$.  
      Then $F_{t}(x)$ is invertible with its inverse $F_{t}^{-1}(x)$ satisfying that for any $t \geq 0$ and $x,y \in \R^{d}$,
      \begin{gather}
            \label{eq:Ftinvlip}
            |F_{t}^{-1}(x)-F_{t}^{-1}(y)|
            \leq
            \frac{1}{1-\theta\Delta{L}}|x-y|,        
            \\\label{eq:Ftinvlipxxx}
            |F_{t}^{-1}(x)|
            \leq
            \big(1 + 2\theta{L}\Delta\big)|x| 
            + C \Delta.          
      \end{gather}
\end{lemma}

\begin{proof}
      Assumption \ref{ass:globallinear} implies that for any $t \geq 0$ and $x,y \in \R^{d}$,
      \begin{align*}
            \big\langle x-y,F_{t}(x)-F_{t}(y) \big\rangle
            =
            |x-y|^{2} - \theta\Delta\big\langle x-y,f(t,x)-f(t,y) \big\rangle
            \geq
            (1-\theta\Delta{L})|x-y|^{2},
      \end{align*}
      and accordingly $|x-y| \leq \frac{1}{1-\theta\Delta{L}} |F_{t}(x)-F_{t}(y)|$. It follows that $F_{t}(x)$ is invertible and
      \begin{equation*}
            |F_{t}^{-1}(x)-F_{t}^{-1}(y)|
            \leq
            \frac{1}{1-\theta\Delta{L}}
            |F_{t}(F_{t}^{-1}(x))-F_{t}(F_{t}^{-1}(y))| 
            =
            \frac{1}{1-\theta\Delta{L}}|x-y|,          
      \end{equation*}
      which proves \eqref{eq:Ftinvlip}. Since $F_{t}(x) = 0$ admits a unique solution $x_{\Delta}(t) := F_{t}^{-1}(0)$ (see, e.g., \cite[Theorem C.2]{stuart1996dynamical}), using $x_{\Delta}(t) = \theta\Delta f(t,x_{\Delta}(t))$ and Assumption \ref{ass:globallinear} yields
      \begin{align*}
            \big|x_{\Delta}(t)\big|^{2}
            =&~
            \theta\Delta \big\langle x_{\Delta}(t), 
            f(t,x_{\Delta}(t)) - f(t,0) \big\rangle
            +
            \theta\Delta \big\langle x_{\Delta}(t), f(t,0) \big\rangle  
            \\\leq&~ 
            \theta\Delta L\big|x_{\Delta}(t)\big|^{2} 
            + C\theta\Delta\big|x_{\Delta}(t)\big|.   
      \end{align*}
      When $x_{\Delta}(t) \neq 0$, we have
      $\big(1 - \theta\Delta L\big)|x_{\Delta}(t)| \leq C\theta\Delta$ and thus
      \begin{align}\label{eq:Ftinv0}
            |F_{t}^{-1}(0)| = |x_{\Delta}(t)| 
            \leq \frac{C\theta\Delta}{1 - \theta\Delta L}
            \leq 
            2 C \theta \Delta 
      \end{align}
      because of $\theta{L}\Delta \leq \frac{1}{2}$. Besides, \eqref{eq:Ftinv0} obviously holds 
      when $x_{\Delta}(t) = 0$. Since $\frac{1}{1-x} \leq 1 + 2x$ for $x \in [0,1/2]$ and $\theta{L}\Delta \leq \frac{1}{2}$, we have $\frac{1}{1-\theta{L}\Delta} \leq 1 + 2\theta{L}\Delta$. Together with \eqref{eq:Ftinvlip} and \eqref{eq:Ftinv0}, we obtain \eqref{eq:Ftinvlipxxx} and complete the proof.
\end{proof}

Furthermore, the following two elementary lemmas regarding algebraic inequalities will be frequently utilized in the subsequent moment estimates.
\begin{lemma}\label{lem:uvp}
      For any $p > 1$ and $\varepsilon \in (0,1]$, there exists a constant $C(p) := \big(\frac{2^{-\frac{p}{p-1}}}{p-1}\big)^{1-p} > 0$ such that for any $u,v \geq 0$,
      \begin{equation*}
            (u+v)^{p}
            \leq
            (1+\varepsilon)u^{p} + C(p)\varepsilon^{-(p-1)}v^{p}.
      \end{equation*}
\end{lemma}

\begin{proof}
      Since the mapping $x \mapsto x^{p}$ is convex for any $x \in \R$ and $p > 1$, we have that for any $u,v \geq 0$ and $\lambda \in (0,1)$, 
      \begin{align*}
            (u+v)^{p}
            =
            \bigg((1-\lambda)\frac{u}{1-\lambda}
            + \lambda \frac{v}{\lambda}\bigg)^{p}
            \leq
            (1-\lambda)^{1-p}u^{p} + \lambda^{1-p}v^{p}.           
      \end{align*}
      For any $\varepsilon \in (0,1]$, letting $(1-\lambda)^{1-p} = 1+\varepsilon$ 
      yields $\lambda = 1 - (1+\varepsilon)^{-\frac{1}{p-1}} \in (0,1)$ because of 
      $(1+\varepsilon)^{-\frac{1}{p-1}} \in (0,1)$. Considering the function $g(\varepsilon)
      := 1 - (1+\varepsilon)^{-\rho}$ for $\varepsilon \in (0,1]$ with $\rho := \frac{1}{p-1} > 0$, 
      the mean value theorem ensures that there exists $\xi \in (0,\varepsilon) \subset [0,1]$ such that
      $g(\varepsilon) - g(0) = g'(\xi)(\varepsilon - 0)$. In view of $g(0) = 0$ and $g'(\xi) = \rho(1+\xi)^{-\rho-1} \geq \rho 2^{-\rho-1}$ due to $\xi \in (0,1)$, 
      we have $g(\varepsilon) \geq \rho 2^{-\rho-1}\varepsilon$. It follows that
      $\lambda^{1-p} \leq (\rho 2^{-\rho-1}\varepsilon)^{1-p} 
      = C(p)\varepsilon^{-(p-1)}$ with $C(p) = (\rho 2^{-\rho-1})^{1-p}$, 
      which completes the proof.
\end{proof}

\begin{lemma}\label{lem:xypCp}
      For any $p \geq 2$, there exists a constant $C(p) := p(p-1)2^{p-3} > 0$ such that
      \begin{align*}
            |x+y|^{p}
            \leq |x|^{p} + p|x|^{p-2}\langle x,y \rangle
            + 
            C(p)\big(|x|^{p-2}|y|^{2} + |y|^{p}\big),
            \quad x,y \in \R^{d}.
      \end{align*}
\end{lemma}

\begin{proof}
      For any $p \geq 2$ and $x,y \in \R^{d}$, letting $\varphi(t) := |x+ty|^{p}, 
      t \in [0,1]$ and using the Taylor formula lead to
      \begin{align*}
            \varphi(1) 
            = \varphi(0) + \varphi'(0) 
              + \int_{0}^{1} (1-s)\varphi''(s) \diff{s}. 
      \end{align*}
      Owing to $\varphi(0) = |x|^{p}, \varphi(1) = |x+y|^{p}$, $\varphi'(t) = p|x+ty|^{p-2}\langle x+ty,y \rangle$ and
      \begin{align*}
            \varphi''(t) 
            =
            p(p-2)|x+ty|^{p-4}(\langle x+ty,y \rangle)^{2}
            +
            p|x+ty|^{p-2}|y|^{2}
            \leq
            p(p-1)|x+ty|^{p-2}|y|^{2},
      \end{align*}      
      we obtain
      \begin{align*}
            |x+y|^{p}
            \leq |x|^{p} + p|x|^{p-2}\langle x,y \rangle
            + 
            p(p-1)|y|^{2}
            \int_{0}^{1} (1-s)|x+sy|^{p-2} \diff{s}.
      \end{align*}
      Since $|x+sy|^{p-2} \leq (|x|+|y|)^{p-2} \leq 2^{p-2}(|x|^{p-2}+|y|^{p-2})$ 
      for any $s \in [0,1]$ and $\int_{0}^{1} (1-s) \diff{s} = \frac{1}{2}$, 
      we immediately give the desired result and finish the proof.
\end{proof}

We are now ready to establish bounded high-order moments for the stochastic theta approximation generated by \eqref{eq:Stmethod}. For later use, we introduce the notation 
\begin{equation}\label{eq:FGHnotation}
      F_{n} := Y_{n} + (1-\theta) f(t_{n},Y_{n})\Delta, \quad
      G_{n} := g(t_{n},Y_{n})\Delta W_{n}, \quad
      H_{n} := \int_{t_{n}}^{t_{n+1}}\int_{|z|<c} 
      h(t_n,Y_n,z) \,\widetilde{N}(\dif{z},\dif{s}),
\end{equation}
so that \eqref{eq:Stmethod} can be written as  
$$Y_{n+1} = \theta f(t_{n+1},Y_{n+1}) \Delta + F_{n} + G_{n} + H_{n}.$$

\begin{proposition}\label{prop:EYnpestimate}
      Suppose that Assumptions \ref{ass:globallinear} and \ref{As:fghgrowthcond} hold, and let $\theta{L}\Delta \leq \frac{1}{2}$ with $\theta \in [0,1]$. Then for any $p \geq 2$, there exists a constant $C > 0$, independent of $n$, such that for any $n = 0,1,\cdots$,
      \begin{align}\label{eq:EYnpestimate}
            \E\big[|Y_{n}|^{p}\big]
            \leq 
            Ce^{Ct_{n}}.
      \end{align}
\end{proposition}

\begin{proof}
      By \eqref{eq:Stmethod}, \eqref{eq:mapFtx}, and \eqref{eq:FGHnotation}, we have $F_{t_{n+1}}(Y_{n+1}) = R_{n}$ with $R_{n} = F_{n} + G_{n} + H_{n}$ for any $n = 0,1,2,\cdots$.
      Moreover, \eqref{eq:Ftinvlipxxx} implies that
      $$|Y_{n+1}| = |F_{t_{n+1}}^{-1}(R_{n})|
      \leq \big(1 + 2\theta{L}\Delta\big)|R_{n}| + C \Delta.$$ Applying Lemma \ref{lem:uvp} with $\varepsilon = \Delta \in (0,1)$ and the inequality $(1+y)^{p} \leq 1 + p2^{p-1}y, y \in [0,1]$ yields
      \begin{align*}
            |Y_{n+1}|^{p}
            \leq
            (1+\Delta)
            \big(1 + 2\theta{L}\Delta\big)^{p}|R_{n}|^{p} 
            + C(p)\Delta^{-(p-1)} (C \Delta)^{p}
            \leq
            (1+C\Delta)|R_{n}|^{p} + C\Delta,
      \end{align*} 
      and thus
      \begin{align}\label{eq:absYnplus1p}
            \E\big[|Y_{n+1}|^{p}\big]
            \leq
            (1+C\Delta)\E\big[|R_{n}|^{p}\big] + C\Delta.
      \end{align}          
      Together with Lemma \ref{lem:xypCp}, we obtain
      \begin{align*}
            \E\big[|R_{n}|^{p}\big]
            =
            \E\big[\E\big(|R_{n}|^{p} \mid \F_{t_{n}} \big)\big]
            \leq&~
            \E\big[\E\big(|F_{n}|^{p} + p|F_{n}|^{p-2}\langle F_{n},G_{n} + H_{n} \rangle
            \\&~+ 
            C(p)\big(|F_{n}|^{p-2}|G_{n} + H_{n}|^{2} + |G_{n} + H_{n}|^{p}\big) \mid \F_{t_{n}} \big)\big].
      \end{align*}
      As $F_{n}$ is $\F_{t_{n}}$-measurable, $\E[G_{n}] = 0$ and $\E[H_{n}] = 0$, one gets
      \begin{align*}
            \E\big(p|F_{n}|^{p-2}\big\langle F_{n},G_{n} + H_{n} \big\rangle \mid \F_{t_{n}} \big)
            =
            p|F_{n}|^{p-2}\big\langle F_{n},\E[G_{n}] 
            + \E[H_{n}] \big\rangle  
            =
            0,
      \end{align*}
      and accordingly
      \begin{align}\label{eq:estimateRnp}
            \E\big[|R_{n}|^{p}\big]
            \leq
            \E\big[|F_{n}|^{p}\big]
            +
            C(p)\E\big[|F_{n}|^{p-2}\E\big(
            |G_{n} + H_{n}|^{2} \mid \F_{t_{n}} \big)\big]
            +
            C(p)\E\big[\E\big(
            |G_{n} + H_{n}|^{p} \mid \F_{t_{n}} \big)\big].
      \end{align} 
      To estimate $\E\big[|F_{n}|^{p}\big]$, using Lemma \ref{lem:xypCp} and Assumption \ref{ass:globallinear} gives
      \begin{align*}
            &~\E\big[|F_{n}|^{p}\big] = \E\big[|Y_{n} + (1-\theta) f(t_{n},Y_{n})\Delta|^{p}\big]
            \\\leq&~
            \E\big[|Y_{n}|^{p} + p|Y_{n}|^{p-2}\langle Y_{n},(1-\theta) f(t_{n},Y_{n})\Delta \rangle
            \\&~+ 
            C(p)\big(|Y_{n}|^{p-2}|(1-\theta) f(t_{n},Y_{n})\Delta|^{2} + |(1-\theta) f(t_{n},Y_{n})\Delta|^{p}\big)\big]
            \\\leq&~
            \E\big[|Y_{n}|^{p}\big]
            +  
            C\Delta\E\big[|Y_{n}|^{p-1}|f(t_{n},Y_{n})|
            + 
            \big(|Y_{n}|^{p-2} |f(t_{n},Y_{n})|^{2} + |f(t_{n},Y_{n})|^{p}\big)\big]
            \\\leq&~
            \E\big[|Y_{n}|^{p}\big]
            +  
            C\Delta\E\big[|Y_{n}|^{p-1}(1 + |Y_{n}|)
            + 
            \big(|Y_{n}|^{p-2} (1 + |Y_{n}|^{2}) 
            + (1 + |Y_{n}|^{p})\big)\big]
            \\\leq&~
            \big(1 + C\Delta\big)\E\big[|Y_{n}|^{p}\big]
            +  
            C\Delta.
      \end{align*} 
      For $C(p)\E\big[|F_{n}|^{p-2}\E\big(|G_{n} + H_{n}|^{2} \mid \F_{t_{n}} \big)\big]$, we utilize
      \begin{align*}
            \E\big(|G_{n}|^{p} \mid \F_{t_{n}} \big) 
            = |g(t_{n},Y_{n})|^{p} \E\big[|\Delta W_{n}|^{p}]
            \leq C\Delta^{\frac{p}{2}}\big(1+|Y_{n}|^{p}\big),          
      \end{align*}
      and 
      \begin{align}\label{eq:Hnpestimate}
            \E\big(|H_{n}|^{p} \mid \F_{t_{n}} \big) \notag
            \leq&~ C(p)\bigg(
            \Big(\int_{t_{n}}^{t_{n+1}}\int_{|z|<c} 
            |h(t_n,Y_n,z)|^{2} \,\mu(\dif{z})\dif{s} \Big)^{\frac{p}{2}} \notag
            \\&~+
            \int_{t_{n}}^{t_{n+1}}\int_{|z|<c} 
            |h(t_n,Y_n,z)|^{p} \,\mu(\dif{z})\dif{s}\bigg) \notag
            \\\leq&~ C(p)\bigg(\Delta^{\frac{p}{2}} 
            \Big(\int_{|z|<c} |h(t_n,Y_n,z)|^{2} 
            \,\mu(\dif{z}) \Big)^{\frac{p}{2}}  \notag
            \\&~+
            \Delta \int_{|z|<c} 
            |h(t_n,Y_n,z)|^{p} \,\mu(\dif{z})\bigg)  \notag
            \\\leq&~ 
            C\Delta \big(1 + |Y_n|^{p}\big),     
      \end{align}  
      as well as the Young inequality to obtain
      \begin{align*}
            C(p)\E\big[|F_{n}|^{p-2}\E\big(
            |G_{n} + H_{n}|^{2} \mid \F_{t_{n}} \big)\big]
            \leq&~
            C \Delta \E\big[|F_{n}|^{p-2}(1 + |Y_n|^{2})\big]
            \\\leq&~
            C \Delta \E\bigg[\frac{p-2}{p}
            (|F_{n}|^{p-2})^{\frac{p}{p-2}}
            + 
            \frac{2}{p}(1 + |Y_n|^{2})^{\frac{p}{2}}\bigg]
            \\\leq&~
            C \Delta \big(1 + \E\big[|Y_n|^{p}\big]\big).
      \end{align*}
      Similarly for $C(p)\E\big[\E\big(|G_{n} + H_{n}|^{p}| \F_{t_{n}} \big)\big]$, one gets
      \begin{align*}
            C(p)\E\big[\E\big(|G_{n} + H_{n}|^{p}| \F_{t_{n}} \big)\big]
            \leq
            C(p)\E\big[\E\big(|G_{n}|^{p}| \F_{t_{n}} \big)
            +
            \E\big(|H_{n}|^{p}| \F_{t_{n}} \big)\big]
            \leq 
            C \Delta \big(1 + \E\big[|Y_n|^{p}\big]\big).
      \end{align*} 
      Inserting these three estimates into \eqref{eq:estimateRnp} leads to 
      $\E\big[|R_{n}|^{p}\big] \leq \big(1 + C\Delta\big)\E\big[|Y_{n}|^{p}\big] + C\Delta$,  
      which together with \eqref{eq:absYnplus1p} yields $\E\big[|Y_{n+1}|^{p}\big] \leq \big(1 + C\Delta\big)\E\big[|Y_{n}|^{p}\big] + C\Delta$, and consequently
      \begin{align*}
            \E\big[|Y_{n}|^{p}\big]
            \leq
            \big(|x_{0}|^{p} + Ct_{n}\big)
            +
            C\Delta \sum_{i=0}^{n-1} \E\big[|Y_{i}|^{p}\big].
      \end{align*}     
      By the discrete Gronwall inequality, we obtain the required result and complete the proof.
\end{proof}

Thanks to Proposition \ref{prop:EYnpestimate}, moments of $\{Y_{n}\}_{n \in \N}$ of arbitrary order are bounded. Therefore, in the estimates below, we shall use 
$q > 0$ as a generic exponent and freely written higher-order polynomial terms such as $|Y_{n}|^{q+1}$ or $|Y_{n}|^{2q}$ in the unified form $(1+|Y_{n}|^{q})$, after changing the value of $q$  if necessary.


\subsection{One-step weak expansion for the frozen Euler step}
In this subsection, we establish a one-step weak consistency expansion for the frozen Euler step   
\begin{align}\label{eq:YEMfrozen}
	  \overline{Y}_{n+1}
	  :=&~
	  Y_{n} 
	  +
	  f(t_{n},Y_{n})\Delta
      +
      g(t_{n},Y_{n})\Delta W_{n}  \notag
      \\&~+
      \int_{t_{n}}^{t_{n+1}}\int_{|z|<c}
      h(t_n,Y_n,z) \,\widetilde{N}(\dif{z},\dif{s}),
      \quad n = 0,1,2,\cdots.
\end{align}  
To this end, we introduce the backward Kolmogorov function
\begin{equation}
      u(s,x) := \E\big[\Phi(Y(t)) \mid Y(s) = x\big],
      \quad 0 \leq s \leq t,x \in \R^{d}.    
\end{equation}
Write $\partial_{s}u$, $\nabla{u}$ and $\nabla^{2}u$ for the time derivative, gradient and Hessian of $u$, respectively. Assume that all components of $f(t,\cdot), g(t,\cdot), h(t,\cdot,z)$ belong to $C_{p}^{4}(\R^{d};\R)$ for any $t \geq 0$ and $z \in \R^{d} \backslash \{0\}$. Then, by \cite[Lemma 12.3.1]{platen2010numerical}, $u$ is the unique solution to the Kolmogorov backward partial integro differential equations (KB-PIDEs)
\begin{equation}\label{eq:Kbpide}
\left\{\begin{aligned}
      &\partial_{s}u(s,x) + \mathcal{L}_{s}u(s,x) = 0,
      \quad (s,x) \in (0,t) \times \R^{d},
      \\
      &u(t,x) = \Phi(x), \quad x \in \R^{d},
\end{aligned}\right.
\end{equation}
where the generator $\mathcal{L}_{s}$ is defined by
\begin{align}\label{eq:mathcalLs}
      \mathcal{L}_{s}\psi(x)
      :=&~
      \big\langle \nabla\psi(x), f(s,x)\big\rangle
      +
      \frac{1}{2}\tr\big(g^{\top}(s,x)\nabla^{2}\psi(x)g(s,x)\big) \notag
      \\&~+
      \int_{|z|<c} \big(\psi(x+h(s,x,z)) 
      - \psi(x) - \langle \nabla\psi(x),h(s,x,z) \rangle\big) \,\mu(\dif{z})
\end{align}
for $\psi \colon \R^{d} \to \R$. Moreover, $x \mapsto u(s,x)$ belongs to $C_{p}^{4}(\R^{d};\R)$ for each $s \in [0,t]$ (see \cite[Lemma 12.3.1]{platen2010numerical}). With the backward Kolmogorov function $u$ at hand, we are able to derive the one-step weak expansion for the frozen Euler step. The following lemma provides the conditional Taylor expansion, together with the corresponding remainder estimate, for the Brownian part of the frozen Euler step.

\begin{lemma}\label{lem:BMpart}
      Assume that all components of $f(t,\cdot), g(t,\cdot), h(t,\cdot,z)$ belong to $C_{p}^{4}(\R^{d};\R)$ for any $t \geq 0$ and $z \in \R^{d} \backslash \{0\}$. 
      Suppose that Assumptions \ref{ass:globallinear}, \ref{As:fghgrowthcond} hold, and let $\theta{L}\Delta \leq \frac{1}{2}$ with $\theta \in [0,1]$. Then we have
      \begin{align*}
            &~\E\big(u\big(t_{n+1},Y_{n} +  f(t_{n},Y_{n})\Delta
            + g(t_{n},Y_{n})\Delta W_{n}\big)  \mid \F_{t_{n}}\big)
            \\=&~
            u(t_{n+1},Y_{n}) 
            + 
            \Delta\big\langle {\nabla} u(t_{n+1},Y_{n}),
            f(t_{n},Y_{n})\big\rangle
            +
            \mathcal{O}\big(\big(1 + |Y_{n}|^{q}\big)\Delta^{2}\big)
            \\&~+
            \frac{\Delta}{2}\tr\big(g(t_{n},Y_{n})^{\top}
            \nabla^{2}u(t_{n+1},Y_{n})g(t_{n},Y_{n})\big).
      \end{align*}
\end{lemma}

\begin{proof}
      Letting $\chi_{n} := f(t_{n},Y_{n})\Delta + g(t_{n},Y_{n})\Delta W_{n}$ and
      $\varphi_{n}(x) := u(t_{n+1},x)$ for $x \in \R^{d}$, the Taylor formula shows that
      \begin{align*}
            &~\varphi_{n}(Y_{n} +  f(t_{n},Y_{n})\Delta + g(t_{n},Y_{n})\Delta W_{n})
            \\=&~
            \varphi_{n}(Y_{n}) + \big\langle \nabla\varphi_{n}(Y_{n}),\chi_{n} \big\rangle
            +
            \frac{1}{2}\chi_{n}^{\top} \nabla^{2}\varphi_{n}(Y_{n}) \chi_{n}
            +
            \frac{1}{6}\nabla^{3}\varphi_{n}(Y_{n})[\chi_{n},\chi_{n},\chi_{n}]
            +
            \hat{R}_{n+1}
      \end{align*}  
      with $\hat{R}_{n+1} := \int_{0}^{1} \frac{(1-r)^{3}}{6}
            \nabla^{4}\varphi_{n}(Y_{n} + r\chi_{n})
            [\chi_{n},\chi_{n},\chi_{n},\chi_{n}]\diff{r}$.
      It follows from $\E(\Delta W_n\mid\F_{t_n}) = 0$ that
      \begin{align}\label{eq:EvarnchiFtn}
            &~\E\big(\varphi_{n}(Y_{n} + f(t_{n},Y_{n})\Delta 
            + g(t_{n},Y_{n})\Delta W_{n}) \mid \F_{t_{n}}\big) \notag
            \\=&~
            \varphi_{n}(Y_{n})
            + 
            \Delta\big\langle \nabla\varphi_{n}(Y_{n}),
            f(t_{n},Y_{n})\big\rangle
            +
            \frac{1}{2}\E\big(\chi_{n}^{\top} \nabla^{2}\varphi_{n}(Y_{n}) 
            \chi_{n} \mid \F_{t_{n}}\big) \notag
            \\&~+
            \frac{1}{6}\E\big(\nabla^{3}\varphi_{n}(Y_{n})
            [\chi_{n},\chi_{n},\chi_{n}] \mid \F_{t_{n}}\big)
            +
            \E\big(\hat{R}_{n+1} \mid \F_{t_{n}}\big).
      \end{align}
      To estimate $\frac{1}{2}\E\big(\chi_{n}^{\top} \nabla^{2}\varphi_{n}(Y_{n}) 
      \chi_{n} \mid \F_{t_{n}}\big)$, the following basic equality 
      \begin{align*}
            \chi_{n}^{\top} \nabla^{2}\varphi_{n}(Y_{n}) \chi_{n}
            =&~
            (f(t_{n},Y_{n})\Delta)^{\top} \nabla^{2}\varphi_{n}(Y_{n}) (f(t_{n},Y_{n})\Delta)
            \\&~+
            2(f(t_{n},Y_{n})\Delta)^{\top} \nabla^{2}\varphi_{n}(Y_{n}) (g(t_{n},Y_{n})\Delta W_{n})
            \\&~+
            (g(t_{n},Y_{n})\Delta W_{n})^{\top} \nabla^{2}\varphi_{n}(Y_{n}) (g(t_{n},Y_{n})\Delta W_{n})
      \end{align*}
      and $\E(\Delta W_n\mid\F_{t_n}) = 0$ give that
      \begin{align*}
            \frac{1}{2}\E\big(\chi_{n}^{\top} &~\nabla^{2}\varphi_{n}(Y_{n}) 
            \chi_{n} \mid \F_{t_{n}}\big)
            =
            \frac{\Delta^{2}}{2} f(t_{n},Y_{n})^{\top} 
            \nabla^{2}\varphi_{n}(Y_{n}) f(t_{n},Y_{n}) 
            \\&~+
            \frac{1}{2}\E\big((\Delta W_{n})^{\top} 
            \big(g(t_{n},Y_{n})^{\top} 
            \nabla^{2}\varphi_{n}(Y_{n}) g(t_{n},Y_{n}) \big)
            (\Delta W_{n}) \mid \F_{t_{n}}\big).
      \end{align*}
      Noting that $\mathbf{u}^{\top}\mathbf{M}\mathbf{u} = \tr(\mathbf{M}\mathbf{u}\mathbf{u}^{\top})$ for any matrix $\mathbf{M} \in \R^{m \times m}$ and random vector $\mathbf{u} \in \R^{m}$, we obtain
      \begin{align*}
            &~\E\big((\Delta W_{n})^{\top} \big(g(t_{n},Y_{n})^{\top} 
            \nabla^{2}\varphi_{n}(Y_{n}) g(t_{n},Y_{n}) \big)
            (\Delta W_{n}) \mid \F_{t_{n}}\big)
            \\=&~
            \E\big(\tr\big(
            \big(g(t_{n},Y_{n})^{\top} 
            \nabla^{2}\varphi_{n}(Y_{n}) g(t_{n},Y_{n}) \big)
            (\Delta W_{n})(\Delta W_{n})^{\top} \big)  \mid \F_{t_{n}}\big)
            \\=&~
            \tr\big(g(t_{n},Y_{n})^{\top} 
            \nabla^{2}\varphi_{n}(Y_{n}) g(t_{n},Y_{n}) 
            \,\E\big((\Delta W_{n})(\Delta W_{n})^{\top} 
            \mid \F_{t_{n}}\big)\big) \big) 
            \\=&~
            \Delta \tr\big(g(t_{n},Y_{n})^{\top} 
            \nabla^{2}\varphi_{n}(Y_{n}) g(t_{n},Y_{n}) \big),  
      \end{align*}
      and accordingly  
      \begin{align}\label{eq:onehalfEchin}
            \frac{1}{2}\E\big(\chi_{n}^{\top} \nabla^{2}\varphi_{n}(Y_{n}) 
            \chi_{n} \mid \F_{t_{n}}\big)
            =&~
            \frac{\Delta^{2}}{2} f(t_{n},Y_{n})^{\top} 
            \nabla^{2}\varphi_{n}(Y_{n}) f(t_{n},Y_{n}) \notag
            \\&~+
            \frac{\Delta}{2} \tr\big(g(t_{n},Y_{n})^{\top} 
            \nabla^{2}\varphi_{n}(Y_{n}) g(t_{n},Y_{n}) \big)
      \end{align}    
      with
      \begin{align}\label{eq:onehalfEchinest}
            \bigg|\frac{\Delta^{2}}{2} f(t_{n},Y_{n})^{\top} 
            \nabla^{2}\varphi_{n}(Y_{n}) f(t_{n},Y_{n}) \bigg|
            \leq 
            C\big(1 + |Y_{n}|^{q}\big)\Delta^{2}
      \end{align}
      because of the linear growth of $f$ and the polynomial growth of $\nabla^{2}\varphi_{n}$.
      For $\frac{1}{6}\E\big(\nabla^{3}\varphi_{n}(Y_{n}) [\chi_{n},\chi_{n},\chi_{n}] \mid \F_{t_{n}}\big)$, we apply 
      \begin{align*}
            [\chi_{n},\chi_{n},\chi_{n}]
            =&~
            (f(t_{n},Y_{n})\Delta)^{\otimes3}
            +
            3(f(t_{n},Y_{n})\Delta)^{\otimes2} \otimes (g(t_{n},Y_{n})\Delta W_{n})
            \\&~+
            3(f(t_{n},Y_{n})\Delta) \otimes (g(t_{n},Y_{n})\Delta W_{n})^{\otimes2}
            +
            (g(t_{n},Y_{n})\Delta W_{n})^{\otimes3}, 
      \end{align*}
      and $\E(\Delta W_n \mid\F_{t_n}) = 0$, $\E((\Delta W_n)^{\otimes3} \mid\F_{t_n}) = 0$ to derive
      \begin{align*}
            \frac{1}{6}\E\big(\nabla^{3}\varphi_{n}(Y_{n})
            &~[\chi_{n},\chi_{n},\chi_{n}] \mid \F_{t_{n}}\big) \notag
            =
            \frac{\Delta^{3}}{6}\nabla^{3}\varphi_{n}(Y_{n})
            [f(t_{n},Y_{n}),f(t_{n},Y_{n}),f(t_{n},Y_{n})] \notag
            \\&~+
            \frac{\Delta}{2}\E\big(\nabla^{3}\varphi_{n}(Y_{n})
            [f(t_{n},Y_{n}),g(t_{n},Y_{n})\Delta W_{n},
            g(t_{n},Y_{n})\Delta W_{n}] \mid \F_{t_{n}}\big),           
      \end{align*}
      which together with the linear growth of $f$ and the polynomial growth of $\nabla^{3}\varphi_{n}$ indicates
      \begin{align}\label{eq:oensixnabla3varnYn}
            \bigg|\frac{1}{6}\E\big(\nabla^{3}\varphi_{n}(Y_{n})
            [\chi_{n},\chi_{n},\chi_{n}] \mid \F_{t_{n}}\big) \bigg|
            \leq 
            C\big(1 + |Y_{n}|^{q}\big)\Delta^{2}.                       
      \end{align}      
      Concerning $\E\big(\hat{R}_{n+1} \mid \F_{t_{n}}\big)$, the linear growth of $f,g$ ensures that for any $p \geq 2$, 
      \begin{align}\label{eq:chinpestimate}
            \E\big(|\chi_{n}|^{p} \mid \F_{t_{n}}\big)
            =
            \E\big(|f(t_{n},Y_{n})\Delta + g(t_{n},Y_{n})\Delta W_{n}|^{p} \mid \F_{t_{n}}\big)
            \leq
            C\big(1 + |Y_{n}|^{q}\big)\Delta^{\frac{p}{2}},
      \end{align}
      which along with the polynomial growth of $\nabla^{4}\varphi_{n}$ results in
      \begin{align}\label{eq:EhatRnplus1}
            \big|\E\big(\hat{R}_{n+1} \mid \F_{t_{n}}\big)\big|
            =&~
            \bigg|\E\bigg(\int_{0}^{1} \frac{(1-r)^{3}}{6}
            \nabla^{4}\varphi_{n}(Y_{n} + r\chi_{n})
            [\chi_{n},\chi_{n},\chi_{n},\chi_{n}]\diff{r} 
            \mid \F_{t_{n}}\bigg)\bigg| \notag
            \\\leq&~
            C\E\big( (1 + |Y_{n}|^{q} + |\chi_{n}|^{q})
            |\chi_{n}|^{4} \mid \F_{t_{n}}\big) \notag
            \\\leq&~
            C\big(1 + |Y_{n}|^{q}\big)\Delta^{2}.
      \end{align}
      Putting \eqref{eq:onehalfEchin}, \eqref{eq:onehalfEchinest}, 
      \eqref{eq:oensixnabla3varnYn} and \eqref{eq:EhatRnplus1} into \eqref{eq:EvarnchiFtn} gives the desired result and 
      ends the proof.
\end{proof}

Similarly, the next lemma assesses the conditional expectation of the weak difference associated with the jump component with respect to the Poisson random measure.
\begin{lemma}\label{lem:Jumppart}
      Assume that all components of $f(t,\cdot), g(t,\cdot), h(t,\cdot,z)$ belong to $C_{p}^{4}(\R^{d};\R)$ for any $t \geq 0$ and $z \in \R^{d} \backslash \{0\}$. Suppose that Assumptions \ref{ass:globallinear}, \ref{As:fghgrowthcond} hold, and let $\theta{L}\Delta \leq \frac{1}{2}$ with $\theta \in [0,1]$. Then we have
      \begin{align}\label{eq:Jumppart}
            &~\E\big(u(t_{n+1},Y_{n} + H_{n}) - u(t_{n+1},Y_{n})
            - \big\langle \nabla u(t_{n+1},Y_{n}),
            H_{n}\big\rangle \mid \F_{t_{n}}\big) \notag
            \\=&~
            \Delta\int_{|z| < c} u(t_{n+1},Y_{n} + h(t_{n},Y_{n},z))
            - u(t_{n+1},Y_{n})  \notag
            \\&~- \big\langle \nabla u(t_{n+1},Y_{n}), 
            h(t_{n},Y_{n},z)\big\rangle \,\mu(\dif{z})
            +
            \mathcal{O}\big(\big(1 + |Y_{n}|^{q}\big)\Delta^{2}\big).
      \end{align}
\end{lemma}

\begin{proof}
      For any $x,y \in \R^{d}$, we define $\varphi_{n}(x) := u(t_{n+1},x)$ 
      and $J(y) := \varphi_{n}(y) - \big\langle \nabla\varphi_{n}(x),y \big\rangle$.
      Since $\varphi_{n} \in C_{p}^{4}(\R^{d};\R)$, there exist $C,q > 0$ such that
      \begin{align*}
            |J(y)| \leq C(1+|x|^{q})(1+|y|^{q}),
            \quad x,y \in \R^{d}. 
      \end{align*}
      In the sequel, we apply this with $x = Y_{n}$, which is $\F_{t_n}$-measurable. In view of
      \begin{gather*}
            J(Y_{n}+H_{n}) - J(Y_{n})
            =
            \varphi_{n}(Y_{n} + H_{n}) - \varphi_{n}(Y_{n})
            - \big\langle \nabla\varphi_{n}(Y_{n}), H_{n}\big\rangle, 
            \\
            J(Y_{n}+h(t_{n},Y_{n},z)) - J(Y_{n})
            =
            \varphi_{n}(Y_{n} + h(t_{n},Y_{n},z))
            - \varphi_{n}(Y_{n})  \notag
            - \big\langle \nabla\varphi_{n}(Y_{n}), 
            h(t_{n},Y_{n},z)\big\rangle,
      \end{gather*}
      proving \eqref{eq:Jumppart} is thus equivalent to proving
      \begin{align}\label{eq:Jumppartequiv}
            &~\E\big(J(Y_{n}+H_{n}) - J(Y_{n}) \mid \F_{t_{n}}\big) \notag
            \\=&~
            \Delta\int_{|z| < c} J(Y_{n}+h(t_{n},Y_{n},z)) - J(Y_{n}) \,\mu(\dif{z})
            +
            \mathcal{O}\big(\big(1 + |Y_{n}|^{q}\big)\Delta^{2}\big).
      \end{align}
      In fact, $\lambda := \mu(\{|z| < c\}) < +\infty$ allows us to define $\overline{\mu}(\dif{z})
      := \frac{1}{\lambda}\mathbf{1}_{\{|z| < c\}}\mu(\dif{z})$ and the Poisson random variable 
      $K_{n} := N(\{|z| < c\} \times (t_{n},t_{n+1}])$ with parameter $\lambda\Delta > 0$. By \cite[Chapter 9]{privault2018understanding}, $K_{n}$ is independent of $\F_{t_n}$ and satisfies
      \begin{align*}
            &~\P(K_{n} = 0) = e^{-\lambda\Delta} = 1-\lambda\Delta + \mathcal{O}(\Delta^{2}),
            \\&~
            \P(K_{n} = 1) = \lambda\Delta e^{-\lambda\Delta} = \lambda\Delta + \mathcal{O}(\Delta^{2}),
            \\&~
            \P(K_{n} \geq 2) = 1 - \P(K_{n} = 0) - \P(K_{n} = 1) = \mathcal{O}(\Delta^{2}).
      \end{align*}
      Besides, given $K_{n} = k \geq 1$, the corresponding jump sizes are denoted by $z_{n,1},z_{n,2},\cdots,z_{n,k}$, which are independent of $\F_{t_n}$, and are independent and identically distributed random variables with the common distribution $\overline{\mu}$ (see, e.g., \cite[Sections 8 and 9 in Chapter 1]{ikeda1989stochastic}). Applying $\widetilde{N}(\dif{z},\dif{s}) := N(\dif{z},\dif{s}) - \mu(\dif{z})\dif{s}$ yields
      \begin{align*}
            H_{n} 
            = 
            \sum_{i=1}^{K_{n}} h(t_n,Y_n,z_{n,i})
            -
            \varrho_{n}\Delta 
      \end{align*}
      with $\varrho_{n} := \int_{|z|<c} h(t_n,Y_n,z) \,\mu(\dif{z})$. Moreover, 
      setting $S_{0} := 0$ and $S_{k} := \sum_{i=1}^{k} h(t_n,Y_n,z_{n,i})$ enables us to rewrite
      $Y_{n}+H_{n} = Y_{n} + S_{K_{n}} - \varrho_{n}\Delta$. From the law of total expectation, it follows that
      \begin{align}\label{eq:Jumppartequivtotal}
            &~\E\big(J(Y_{n}+H_{n}) - J(Y_{n}) \mid \F_{t_{n}}\big) \notag
            \\=&~
            \sum_{k=0}^{\infty}\P(K_{n} = k) \notag
            \E\big(J(Y_{n} + S_{K_{n}} - \varrho_{n}\Delta) - J(Y_{n}) \mid \F_{t_{n}}, K_{n} = k\big)
            \\=&~
            \sum_{k=0}^{\infty}\P(K_{n} = k)
            \E\big(J(Y_{n} + S_{k} - \varrho_{n}\Delta) - J(Y_{n}) \mid \F_{t_{n}}\big).
      \end{align}
      To proceed, we will split the discussion of \eqref{eq:Jumppartequivtotal} into three cases: 
      $k = 0$, $k = 1$, and $k \geq 2$. 
      
      \textbf{Case 1: $k = 0$.} Using $S_{0} = 0$, the Taylor formula, and $\nabla J(x) = 0,x \in \R^{d}$ shows that
      \begin{align*}
            &~\E\big(J(Y_{n} + S_{0} - \varrho_{n}\Delta) - J(Y_{n}) \mid \F_{t_{n}}\big)
            =
            \E\big(J(Y_{n} - \varrho_{n}\Delta) - J(Y_{n}) \mid \F_{t_{n}}\big)
            \\=&~
            \E\bigg( \big\langle \nabla J(Y_{n}), -\varrho_{n}\Delta \big\rangle 
            +
            \frac{1}{2}(-\varrho_{n}\Delta)^{\top} \nabla^{2}J(Y_{n})
            (-\varrho_{n}\Delta)
            \\&~+
            \int_{0}^{1}\frac{(1-r)^{2}}{2}  
            \nabla^{3}J(Y_{n} - r\varrho_{n}\Delta)
            [-\varrho_{n}\Delta,-\varrho_{n}\Delta,-\varrho_{n}\Delta]
            \diff{r} \,\Big|\, \F_{t_{n}}\bigg)
            \\=&~
            \frac{\Delta^{2}}{2} \varrho_{n}^{\top} 
            \nabla^{2}J(Y_{n}) \varrho_{n}
            +
            \int_{0}^{1}\frac{(1-r)^{2}}{2}  
            \nabla^{3}J(Y_{n} - r\varrho_{n}\Delta)
            [-\varrho_{n}\Delta,-\varrho_{n}\Delta,-\varrho_{n}\Delta] \diff{r}.
      \end{align*}
      Then the linear growth of $h$ and the polynomial growth 
      of both $\nabla^{2}J$ and $\nabla^{3}J$ further promise
      \begin{align}\label{eq:estimate111}
            \bigg|\frac{\Delta^{2}}{2} \varrho_{n}^{\top} 
            \nabla^{2}J(Y_{n}) \varrho_{n}\bigg|
            \leq
            \frac{\Delta^{2}}{2} \big|\varrho_{n}\big|^{2} 
            \big|\nabla^{2}J(Y_{n})\big|
            \leq
            C\big(1 + |Y_{n}|^{q}\big)\Delta^{2},            
      \end{align}
      and
      \begin{align}\label{eq:estimate222}
            &~\bigg|\int_{0}^{1}\frac{(1-r)^{2}}{2}  \notag
            \nabla^{3}J(Y_{n} - r\varrho_{n}\Delta)
            [-\varrho_{n}\Delta,-\varrho_{n}\Delta,-\varrho_{n}\Delta]\diff{r}\bigg|
            \\\leq&~
            C\big(1 + |Y_{n}|^{q} + |\varrho_{n}\Delta|^{q}\big) \notag
            |\varrho_{n}\Delta|^{3}  
            \\\leq&~
            C\big(1 + |Y_{n}|^{q}\big)\Delta^{2},          
      \end{align}
      which together with $\P(K_{n} = 0) = 1-\lambda\Delta + \mathcal{O}(\Delta^{2}) = 1 + \mathcal{O}(\Delta)$ indicates
      \begin{align}\label{eq:kcaseone}
            &~\P(K_{n} = 0)\E\big(J(Y_{n} + S_{0} 
            - \varrho_{n}\Delta) - J(Y_{n}) \mid \F_{t_{n}}\big)
            =
            \mathcal{O}\big(\big(1 + |Y_{n}|^{q}\big)\Delta^{2}\big).
      \end{align}
      
      \textbf{Case 2: $k = 1$.} By $S_{1} = h(t_n,Y_n,z_{n,1})$, we have
      \begin{align}\label{eq:kcasetwo0}
            &~\E\big(J(Y_{n} + S_{1} - \varrho_{n}\Delta) 
            - J(Y_{n}) \mid \F_{t_{n}}\big) \notag
            \\=&~
            \E\big(J(Y_{n} + h(t_n,Y_n,z_{n,1}) - \varrho_{n}\Delta) 
            - J(Y_{n} + h(t_n,Y_n,z_{n,1})) \mid \F_{t_{n}}\big) \notag
            \\&~+
            \E\big(J(Y_{n} + h(t_n,Y_n,z_{n,1})) - J(Y_{n}) \mid \F_{t_{n}}\big).
      \end{align}
      On the one hand, the Taylor formula and the arguments used in \eqref{eq:estimate111} and \eqref{eq:estimate222}
      give
      \begin{align*}
            &~\E\big(J(Y_{n} + h(t_n,Y_n,z_{n,1}) - \varrho_{n}\Delta)
            - J(Y_{n} + h(t_n,Y_n,z_{n,1})) \mid \F_{t_{n}}\big)
            \\=&~
            \E\bigg(\big\langle \nabla J(Y_{n} + h(t_n,Y_n,z_{n,1})), 
            - \varrho_{n}\Delta\big\rangle
            +
            \frac{1}{2}(-\varrho_{n}\Delta)^{\top}
            \nabla^{2}J(Y_{n} + h(t_n,Y_n,z_{n,1}))
            (-\varrho_{n}\Delta)
            \\&~+
            \int_{0}^{1}\frac{(1-r)^{2}}{2} 
            \nabla^{3}J(Y_{n} + h(t_n,Y_n,z_{n,1})
            -r\varrho_{n}\Delta)
            [-r\varrho_{n},-r\varrho_{n},-r\varrho_{n}]
            \diff{r} \,\Big|\, \F_{t_{n}}\bigg)
            \\=&~
            \Delta\E\big(\big\langle \nabla J(Y_{n} + h(t_n,Y_n,z_{n,1})), 
            - \varrho_{n}\big\rangle \mid \F_{t_{n}}\big)
            +
            \mathcal{O}\big(\big(1 + |Y_{n}|^{q}\big)\Delta^{2}\big).
      \end{align*}
      Noting that $\P(K_{n} = 1) = \lambda\Delta + \mathcal{O}(\Delta^{2}) = \mathcal{O}(\Delta)$, 
      one can use the linear growth of $h$ and the polynomial growth of $\nabla{J}$ to deduce 
      \begin{align*}
            \big|\P(K_{n} = 1)\Delta\E\big(\big\langle \nabla J(Y_{n} + h(t_n,Y_n,z_{n,1})), 
            - \varrho_{n}\big\rangle \mid \F_{t_{n}}\big)\big|
            \leq
            C\big(1 + |Y_{n}|^{q}\big)\Delta^{2},
      \end{align*}
      and accordingly
      \begin{align}\label{eq:kcasetwo1}
            &~\P(K_{n} = 1)
            \E\big(J(Y_{n} + h(t_n,Y_n,z_{n,1}) - \varrho_{n}\Delta)
            - J(Y_{n} + h(t_n,Y_n,z_{n,1})) \mid \F_{t_{n}}\big) \notag
            \\=&~
            \mathcal{O}\big(\big(1 + |Y_{n}|^{q}\big)\Delta^{2}\big).
      \end{align}
      On the other hand, utilizing the property of conditional expectation (see, e.g., \cite[Lemma 9.2, Chapter 2]{mao2008stochastic}) and the fact that $z_{n,1}$ is independent of $\F_{t_{n}}$ and admits the distribution $\overline{\mu}(\dif{z}) := \frac{1}{\lambda}\mathbf{1}_{\{|z| < c\}}\mu(\dif{z})$ leads to
      \begin{align*}
            &~\E\big(J(Y_{n} + h(t_n,Y_n,z_{n,1})) 
            - J(Y_{n}) \mid \F_{t_{n}}\big)
            \\=&~
            \E\big[J(x + h(t_n,x,z_{n,1})) - J(x)\big]|_{x = Y_{n}}
            \\=&~
            \int_{|z| < c} \big(J(Y_{n} + h(t_n,Y_{n},z)) - J(Y_{n})\big) 
            \,\overline{\mu}(\dif{z}) 
            \\=&~
            \frac{1}{\lambda}\int_{|z| < c} 
            \big(J(Y_{n} + h(t_n,Y_{n},z)) - J(Y_{n})\big) \,\mu(\dif{z}).         
      \end{align*}
      By $\P(K_{n} = 1) = \lambda\Delta + \mathcal{O}(\Delta^{2})$,
      the linear growth of $h$ and the polynomial growth of $J$, one obtains
      \begin{align*}
            &~\bigg|C\frac{\Delta^{2}}{\lambda}\int_{|z| < c} 
            J(Y_{n} + h(t_n,Y_{n},z)) - J(Y_{n}) \,\mu(\dif{z})\bigg|
            \\\leq&~
            C\Delta^{2}\int_{|z| < c} 
            \big(1 + |Y_{n}|^{q} + |h(t_n,Y_{n},z)|^{q}\big) 
            + \big(1 + |Y_{n}|^{q}\big) \,\mu(\dif{z})
            \\\leq&~
            C\Delta^{2}\big(1 + |Y_{n}|^{q}\big),
      \end{align*}
      and thus
      \begin{align}\label{eq:kcasetwo2}
            &~\P(K_{n} = 1)\E\big(J(Y_{n} + h(t_n,Y_n,z_{n,1})) 
            - J(Y_{n}) \mid \F_{t_{n}}\big) \notag
            \\=&~
            \Delta\int_{|z| < c} 
            \big(J(Y_{n} + h(t_n,Y_{n},z)) - J(Y_{n})\big) \,\mu(\dif{z})
            +
            \mathcal{O}\big(\big(1 + |Y_{n}|^{q}\big)\Delta^{2}\big).        
      \end{align}
      Combining \eqref{eq:kcasetwo0}, \eqref{eq:kcasetwo1} and \eqref{eq:kcasetwo2} implies that
      \begin{align}\label{eq:kcasetwo}
            &~\P(K_{n} = 1)\E\big(J(Y_{n} + S_{1} - \varrho_{n}\Delta) 
            - J(Y_{n}) \mid \F_{t_{n}}\big) \notag
            \\=&~
            \Delta\int_{|z| < c} 
            \big(J(Y_{n} + h(t_n,Y_{n},z)) - J(Y_{n})\big) \,\mu(\dif{z})
            +
            \mathcal{O}\big(\big(1 + |Y_{n}|^{q}\big)\Delta^{2}\big).
      \end{align}
      
      \textbf{Case 3: $k \geq 2$.} Recall that the random variables $z_{n,1},z_{n,2},\cdots,z_{n,k}$ are independent of $\F_{t_n}$, and are independent and identically distributed with the common distribution $\overline{\mu}(\dif{z}) := \frac{1}{\lambda}\mathbf{1}_{\{|z| < c\}}\mu(\dif{z})$. Applying  Assumption \ref{As:fghgrowthcond} and $S_{k} := \sum_{i=1}^{k} h(t_n,Y_n,z_{n,i})$ yields
      \begin{align*}
            \E\big(|S_{k}|^{q} \mid \F_{t_{n}} \big)
            \leq&~
            k^{q-1}\sum_{i=1}^{k}
            \E\big(|h(t_n,Y_n,z_{n,i})|^{q} \mid \F_{t_{n}} \big)
            =
            k^{q}\E\big(| h(t_n,Y_n,z_{n,1})|^{q} \mid \F_{t_{n}} \big)
            \\=&~
            k^{q}\int_{|z| < c} |h(t_n,Y_n,z)|^{q} \,\overline{\mu}(\dif{z})
            =
            \frac{k^{q}}{\lambda}\int_{|z| < c} |h(t_n,Y_n,z)|^{q} \,\mu(\dif{z})
            \\\leq&~
            Ck^{q}(1 + |Y_{n}|^{q}).
      \end{align*}
      Due to the polynomial growth of $J$ and Assumption \ref{As:fghgrowthcond}, we have
      \begin{align*}
            &~\big|\E\big(J(Y_{n} + S_{k} - \varrho_{n}\Delta) 
            - J(Y_{n}) \mid \F_{t_{n}}\big)\big|
            \\\leq&~
            C\E\big(\big( 1 + |Y_{n}|^{q} + |S_{k}|^{q} 
            + |\varrho_{n}|^{q}\Delta^{q}\big) 
            + \big(1 + |Y_{n}|^{q}\big) \mid \F_{t_{n}}\big)
            \\\leq&~
            C\big( 1 + |Y_{n}|^{q} 
            + |\varrho_{n}|^{q}\Delta^{q}\big)
            +
            C\E\big(|S_{k}|^{q} \mid \F_{t_{n}} \big)
            \\\leq&~
            C(1+k^{q})\big(1 + |Y_{n}|^{q}\big),
      \end{align*} 
      where $C$ does not depend on $k$. It follows from $\sum_{k=2}^{\infty}
            (1+k^{q})\frac{(\lambda\Delta)^{k-2}}{k!} < \infty$ that
      \begin{align}\label{eq:kcasethree}
            &~\bigg|\sum_{k=2}^{\infty}\P(K_{n} = k)
            \E\big(J(Y_{n} + S_{k} - \varrho_{n}\Delta) 
            - J(Y_{n}) \mid \F_{t_{n}}\big) \bigg| \notag
            \\\leq&~
            C\big(1 + |Y_{n}|^{q}\big)
            e^{-\lambda\Delta}\sum_{k=2}^{\infty}
            (1+k^{q})\frac{(\lambda\Delta)^{k}}{k!}  \notag
            \\\leq&~
            C\big(1 + |Y_{n}|^{q}\big)
            (\lambda\Delta)^{2}\sum_{k=2}^{\infty}
            (1+k^{q})\frac{(\lambda\Delta)^{k-2}}{k!}  \notag
            \\\leq&~
            C\big(1 + |Y_{n}|^{q}\big)\Delta^{2}.
      \end{align}
      Substituting \eqref{eq:kcaseone}, \eqref{eq:kcasetwo} and
      \eqref{eq:kcasethree} into \eqref{eq:Jumppartequivtotal}  
      gives \eqref{eq:Jumppartequiv} and accordingly \eqref{eq:Jumppart}. 
\end{proof}

To analyze the regularity of solution $u$ to \eqref{eq:Kbpide} and bound the time discretization errors of the generator, we introduce the following H\"{o}lder continuity with respect to the time variable.

\begin{assumption}\label{ass:fghtimeLipschitz}
      There exist constants $C > 0$ and $q_{0} \geq 1$ such that for any $s,t \geq 0$ and $x \in \R^{d}$,
      \begin{align*}
            |f(t,x) - f(s,x)|^{2} + |g(t,x) - g(s,x)|^{2}
            +
            \int_{|z| < c} |h(t,x,z) - h(s,x,z)|^{2} \,\mu(\dif{z})
            \leq
            C(1+|x|^{q_{0}})|t-s|^{2}.
      \end{align*} 
\end{assumption}

Under Assumption \ref{ass:fghtimeLipschitz} and the regularity on $f, g, h$, one can prove the time-Lipschitz continuity of $u, \nabla u$ and $\nabla^{2}u$, i.e., there exist constants $C > 0$ and $q \geq 1$ such that for any $s,t \geq 0$ and $x \in \R^{d}$,
\begin{align}\label{eq:utimelips}
      |u(t,x) - u(s,x)| + |\nabla u(t,x) - \nabla u(s,x)|
      +
      |\nabla^{2}u(t,x) - \nabla^{2}u(s,x)|
      \leq
      C(1+|x|^{q})|t-s|.
\end{align}   
%
More precisely, differentiating the KB-PIDE $\partial_t u(t,x)=-\mathcal L_tu(t,x)$ and using
$u \in C_p^{1,4}$ (so that spatial derivatives commute with $\mathcal L_t$), we have
$\partial_t D_x^\alpha u(t,x)=-D_x^\alpha(\mathcal L_tu(t,\cdot))(x)$ for $|\alpha|\le2$.
Moreover, for any $t\in[0,T]$, $|D_x^\alpha(\mathcal L_tu)(t,x)|\le C(1+|x|^q), |\alpha| \leq 2$,
where $C$ is independent of $t$. Hence \eqref{eq:utimelips} follows by integrating in time. The following two lemmas explicitly quantify the time discretization errors associated with  $\mathcal{L}_s$, which are crucial for controlling the integral remainders in the subsequent weak expansion.

\begin{lemma}\label{lem:mathcalLslip1}
      Assume that all components of $f(t,\cdot), g(t,\cdot), h(t,\cdot,z)$ belong to $C_{p}^{4}(\R^{d};\R)$ for any $t \geq 0$ and $z \in \R^{d} \backslash \{0\}$. Suppose that Assumptions \ref{ass:globallinear} and \ref{ass:fghtimeLipschitz} hold. Then there exists a constant $C > 0$ and an integer $q \geq 1$ such that for any $s \in [t_{n},t_{n+1}]$, 
      \begin{align}
            \big|\mathcal{L}_{s}u(t_{n+1},x) 
            - \mathcal{L}_{s}u(s,x)\big|
            \leq
            C(1+|x|^{q})|t_{n+1}-s|.
      \end{align}
\end{lemma}

\begin{proof}
      Applying \eqref{eq:mathcalLs} yields the following decomposition
      \begin{align*}
            &~\mathcal{L}_{s}u(t_{n+1},x)-\mathcal{L}_{s}u(s,x)
            \\=&~
            \big\langle {\nabla}u(t_{n+1},x) - {\nabla}u(s,x), f(s,x)\big\rangle
            +
			\frac{1}{2}\tr\big(g(s,x)g^{\top}(s,x)
            (\nabla^{2}u(t_{n+1},x)-\nabla^{2}u(s,x))\big)
            \\&~+
			\int_{|z|<c} \big(u(t_{n+1},x+h(s,x,z)) - u(s,x+h(s,x,z))\big)
            - \big(u(t_{n+1},x) - u(s,x)\big)
			\\&~- \langle {\nabla}u(t_{n+1},x) - \nabla{u}(s,x),h(s,x,z) \rangle \,\mu(\dif{z}) 
            =:
            A_{1} + A_{2} + A_{3}.         
      \end{align*}
      The linear growth of $f$ and \eqref{eq:utimelips} show that
      \begin{align*}
            |A_{1}|
            \leq
            |{\nabla}u(t_{n+1},x) - {\nabla}u(s,x)| |f(s,x)|
            \leq
            C(1+|x|^{q})|t_{n+1}-s|.
      \end{align*}
      By the linear growth of $g$ and \eqref{eq:utimelips}, one gets
      \begin{align*}
            |A_{2}|
            \leq&~
            \frac{1}{2}\big|\tr\big(g(s,x)g^{\top}(s,x)
            (\nabla^{2}u(t_{n+1},x)-\nabla^{2}u(s,x))\big)\big|
            \\\leq&~
            \frac{1}{2}|g(s,x)|^{2}
            |\nabla^{2}u(t_{n+1},x)-\nabla^{2}u(s,x)|
            \\\leq&~
            C(1+|x|^{q})|t_{n+1}-s|.
      \end{align*}
      Using the linear growth of $h$, \eqref{eq:utimelips}, and $\mu(\{|z| < c\}) < +\infty$ yields
      \begin{align*}
            |A_{3}|
            \leq&~
            \int_{|z|<c} \big|u(t_{n+1},x+h(s,x,z)) - u(s,x+h(s,x,z))\big|
            + \big|u(t_{n+1},x) - u(s,x)\big|
			\\&~+ |{\nabla}u(t_{n+1},x) - \nabla{u}(s,x)|\,|h(s,x,z)|\,\mu(\dif{z}) 
            \\\leq&~
            C\int_{|z|<c} \big(1 + |x+h(s,x,z)|^{q}\big)\,|t_{n+1}-s|
            + \big(1 + |x|^{q}\big)\,|t_{n+1}-s|
			\\&~+ \big(1 + |x|^{q}\big)\,|t_{n+1}-s|\,|h(s,x,z)|\,\mu(\dif{z})
            \\\leq&~
            C|t_{n+1}-s|\int_{|z|<c} \big(1 + |x|^{q} + |h(s,x,z)|^{q}\big)\,\mu(\dif{z})
            \\\leq&~
            C(1+|x|^{q})|t_{n+1}-s|.
      \end{align*}
      Then we obtain the desired result and end the proof.
\end{proof}

\begin{lemma}\label{lem:mathcalLslip2}
      Assume that all components of $f(t,\cdot), g(t,\cdot), h(t,\cdot,z)$ belong to $C_{p}^{4}(\R^{d};\R)$ for any $t \geq 0$ and $z \in \R^{d} \backslash \{0\}$. Suppose that Assumptions \ref{ass:globallinear} and \ref{ass:fghtimeLipschitz} hold. Then there exists constants $C > 0$ and $q \geq 1$ such that for any $s \in [t_{n},t_{n+1}]$, 
      \begin{align*}
            \big|\mathcal{L}_{s}u(t_{n+1},x) - \mathcal{L}_{t_n}u(t_{n+1},x)\big|
            \leq
            C(1+|x|^{q})|s-t_{n}|.
      \end{align*}
\end{lemma}

\begin{proof}
       From \eqref{eq:mathcalLs}, it follows that
      \begin{align*}
            &~\mathcal{L}_{s}u(t_{n+1},x)
            - \mathcal{L}_{t_n}u(t_{n+1},x)
            \\=&~
            \big\langle {\nabla}u(t_{n+1},x), f(s,x)- f(t_n,x)\big\rangle
			\\&~+
			\frac{1}{2}\tr\big(g^{\top}(s,x)
            \nabla^{2}u(t_{n+1},x)g(s,x)
            - g^{\top}(t_n,x)\nabla^{2}u(t_{n+1},x)g(t_n,x)\big)
			\\&~+
			\int_{|z|<c} \big(u(t_{n+1},x+h(s,x,z)) 
			- u(t_{n+1},x+h(t_n,x,z))\big) 
           \\&~-
           \big\langle {\nabla}u(t_{n+1},x),
           h(s,x,z) - h(t_n,x,z) \big\rangle \,\mu(\dif{z})
           =:
           B_{1} + B_{2} + B_{3}.
      \end{align*}      
      By Assumption \ref{ass:fghtimeLipschitz} and $u(t_{n+1},\cdot) \in C_{p}^{4}(\R^{d};\R)$, we have
      \begin{align*}
            |B_{1}|
            \leq
            |{\nabla}u(t_{n+1},x)|\,|f(s,x)- f(t_n,x)|
            \leq
            C(1+|x|^{q})|s-t_{n}|.
      \end{align*} 
      Making use of the linear growth of $g$, Assumption \ref{ass:fghtimeLipschitz} 
      and $u(t_{n+1},\cdot) \in C_{p}^{4}(\R^{d};\R)$ leads to
      \begin{align*}
            |B_{2}|
            \leq&~            
            \frac{1}{2}\big|\tr\big((g(s,x) - g(t_n,x))^{\top}
            \nabla^{2}u(t_{n+1},x)g(s,x)\big)\big|
            \\&~+
            \frac{1}{2}\big|\tr\big(g^{\top}(t_n,x)
            \nabla^{2}u(t_{n+1},x)(g(s,x)-g(t_n,x))\big)\big|
            \\\leq&~            
            C|(g(s,x) - g(t_n,x)|\,|\nabla^{2}u(t_{n+1},x)|\,|g(s,x)|
            \\&~+
            C|g(t_n,x)|\,|\nabla^{2}u(t_{n+1},x)|\,|g(s,x)-g(t_n,x)|
            \\\leq&~
            C(1+|x|^{q})|s-t_{n}|.
      \end{align*}
      Owing to $u(t_{n+1},\cdot) \in C_{p}^{4}(\R^{d};\R)$, 
      \begin{align*}
            |B_{3}|
            \leq&~
            \int_{|z|<c} |u(t_{n+1},x+h(s,x,z)) 
			- u(t_{n+1},x+h(t_n,x,z))| 
           \\&~+
           |{\nabla}u(t_{n+1},x)|\,
           |h(s,x,z) - h(t_n,x,z)|\,\mu(\dif{z})
           \\\leq&~
           C\int_{|z|<c} \big(1 + |x+h(s,x,z)|^{q} 
           + |x+h(t_n,x,z)|^{q}\big)
           |h(s,x,z) - h(t_n,x,z)| \,\mu(\dif{z})
           \\&~+
           \int_{|z|<c} |{\nabla}u(t_{n+1},x)|\,
           |h(s,x,z) - h(t_n,x,z)|\,\mu(\dif{z}).
      \end{align*}
      Together with the H\"{o}lder inequality, $u(t_{n+1},\cdot) \in C_{p}^{4}(\R^{d};\R)$, the linear growth of $h$, $\mu(\{|z| < c\}) < +\infty$, and Assumption \ref{ass:fghtimeLipschitz}, we deduce that
      \begin{align*}
            |B_{3}|
            \leq&~
            C\bigg(\int_{|z|<c} \big(1 + |x+h(s,x,z)|^{q} 
            + |x+h(t_n,x,z)|^{q}\big)^{2} \,\mu(\dif{z})\bigg)^{\frac{1}{2}}
            \\&~\times
            \bigg(\int_{|z|<c} |h(s,x,z) - h(t_n,x,z)|^{2} 
            \,\mu(\dif{z})\bigg)^{\frac{1}{2}}
            \\&~+
            |{\nabla}u(t_{n+1},x)|\,\int_{|z|<c} 
            |h(s,x,z) - h(t_n,x,z)|\,\mu(\dif{z})
            \\\leq&~
            C(1+|x|^{q})|s-t_{n}|.
      \end{align*}   
      As a consequence of estimates on $|B_{1}|$, $|B_{2}|$, and $|B_{3}|$, 
      the required result holds.   
\end{proof}

Consolidating the aforementioned estimates for the Brownian term, the jump term, and the associated generator, we now state the one-step weak consistency result for the frozen Euler approximation.
\begin{proposition}\label{prop:EMerror}
      Assume that all components of $f(t,\cdot), g(t,\cdot), h(t,\cdot,z)$ belong to $C_{p}^{4}(\R^{d};\R)$ for any $t \geq 0$ and $z \in \R^{d} \backslash \{0\}$. Suppose that Assumptions \ref{ass:globallinear}, \ref{As:fghgrowthcond}, \ref{ass:fghtimeLipschitz} hold, and let $\theta{L}\Delta \leq \frac{1}{2}$ with $\theta \in [0,1]$. Then there exists a constant $C > 0$ such that
      \begin{align*}
            \big|\E\big( u(t_{n+1},\overline{Y}_{n+1}) 
            - u(t_{n},Y_{n}) \mid \F_{t_{n}}\big)\big|
            \leq 
            C\big(1 + |Y_{n}|^{q}\big)\Delta^{2}.
      \end{align*}
\end{proposition}

\begin{proof}
      Letting $\F_{t_n}^{H_{n}} := \sigma(\F_{t_n}, \sigma(H_{n}))$ yields $\F_{t_n} \subset \F_{t_n}^{H_{n}} \subset \F$. Setting $\varphi_{n}(x) := u(t_{n+1},x)$, the property of conditional expectation enables us to get  
      \begin{align*}
            \E\big(\varphi_{n}(\overline{Y}_{n+1}) \mid \F_{t_n}\big)
            =&~
            \E\big( \E\big(\varphi_{n}(\overline{Y}_{n+1}) \mid \F_{t_n}^{H_{n}}\big)
            \mid \F_{t_n}\big)
            =
            \E\big( \E\big(\varphi_{n}\big((Y_{n}+H_{n}) + \chi_{n}\big) \mid \F_{t_n}^{H_{n}}\big)
            \mid \F_{t_n}\big)
      \end{align*}
      with $\chi_{n} := f(t_{n},Y_{n})\Delta + g(t_{n},Y_{n})\Delta W_{n}$. Since $Y_{n}+H_{n}$ is $\F_{t_n}^{H_{n}}$-measurable, one can adopt the same techniques used in Lemma \ref{lem:BMpart} to deduce that
      \begin{align*}
            \E\big(\varphi_{n}\big((Y_{n}+H_{n}) + \chi_{n}\big) 
            \mid \F_{t_n}^{H_{n}}\big)
            =&~
            \varphi_{n}(Y_{n}+H_{n}) 
            + 
            \Delta\big\langle \nabla\varphi_{n}(Y_{n}+H_{n}),
            f(t_{n},Y_{n})\big\rangle
            \\&~+
            \frac{\Delta}{2}\tr\big(g(t_{n},Y_{n})^{\top}
            \nabla^{2}\varphi_{n}(Y_{n}+H_{n})g(t_{n},Y_{n})\big)
            +
            R_{n}^{B}.
      \end{align*}
      with the remainder
      \begin{align}\label{eq:Rnbterms} 
            R_{n}^{B}
            :=
            \int_{0}^{1} \frac{(1-r)^{3}}{6}
            \nabla^{4}\varphi_{n}(Y_{n}+H_{n} + r\chi_{n})
            [\chi_{n},\chi_{n},\chi_{n},\chi_{n}]\diff{r}.
      \end{align}
      It follows that
      \begin{align}\label{eq:varphibarYnftn}
            \E\big(\varphi_{n}(\overline{Y}_{n+1}) &~\mid \F_{t_n}\big) \notag
            =
            \E\big(\varphi_{n}(Y_{n}+H_{n}) \mid \F_{t_n}\big)
            + 
            \Delta\E\big(\big\langle \nabla\varphi_{n}(Y_{n}+H_{n}),
            f(t_{n},Y_{n})\big\rangle \mid \F_{t_n}\big) \notag
            \\&~+
            \frac{\Delta}{2}\E\big(\tr\big(g(t_{n},Y_{n})^{\top}
            \nabla^{2}\varphi_{n}(Y_{n}+H_{n})g(t_{n},Y_{n})\big)
            \mid \F_{t_n}\big)
            +
            \E\big(R_{n}^{B} \mid \F_{t_n}\big).
      \end{align}
      For $\E\big(\varphi_{n}(Y_{n}+H_{n}) \mid \F_{t_n}\big)$, Lemma \ref{lem:Jumppart},
      $\E(H_n \mid \F_{t_n}) = 0$ and $\E\big( \big\langle \nabla\varphi_{n}(Y_{n}),
            H_{n}\big\rangle \mid \F_{t_{n}}\big) = 0$ imply
      \begin{align}\label{eq:EvarphinYnplusHn}
            \E\big(\varphi_{n}(Y_{n} + H_{n}) \mid \F_{t_{n}}\big) \notag
            =&~
            \varphi_{n}(Y_{n}) 
            +
            \Delta\int_{|z| < c} \Big(\varphi_{n}(Y_{n} + h(t_{n},Y_{n},z))
            - \varphi_{n}(Y_{n})
            \\&~- \big\langle \nabla\varphi_{n}(Y_{n}),
            h(t_{n},Y_{n},z)\big\rangle\Big) \,\mu(\dif{z})
            +
            \mathcal{O}\big(\big(1 + |Y_{n}|^{q}\big)\Delta^{2}\big).
      \end{align}
      As for $\Delta\E\big(\big\langle \nabla\varphi_{n}(Y_{n}+H_{n}),
      f(t_{n},Y_{n})\big\rangle \mid \F_{t_n}\big)$, applying the Taylor 
      formula and  $\E(H_n \mid \F_{t_n}) = 0$ leads to
      \begin{align}\label{eq:DelatvarphinYnHn}
            &~\Delta\E\big(\big\langle \nabla\varphi_{n}(Y_{n}+H_{n}),
            f(t_{n},Y_{n})\big\rangle \mid \F_{t_n}\big) \notag
            \\=&~
            \Delta\E\bigg(\Big\langle \nabla\varphi_{n}(Y_{n})
            +
            \nabla^{2}\varphi_{n}(Y_{n})H_{n}
            +
            \int_{0}^{1} (1-r)\nabla^{3}\varphi_{n}(Y_{n}+rH_{n}) H_{n}H_{n} \diff{r},
            f(t_{n},Y_{n})\Big\rangle \,\Big|\, \F_{t_n}\bigg) \notag
            \\=&~
            \Delta\big\langle \nabla\varphi_{n}(Y_{n}),
            f(t_{n},Y_{n})\big\rangle 
            +
            \mathcal{O}\big(\big(1 + |Y_{n}|^{q}\big)\Delta^{2}\big)
      \end{align}
      due to
      \begin{align*}
            &~\bigg|\Delta\E\bigg(\Big\langle \int_{0}^{1} (1-r)
            \nabla^{3}\varphi_{n}(Y_{n}+rH_{n}) H_{n}H_{n} \diff{r},
            f(t_{n},Y_{n})\Big\rangle \,\Big|\, \F_{t_n}\bigg)\bigg|
            \\\leq&~
            \Delta |f(t_{n},Y_{n})| \E\bigg( \int_{0}^{1}  
            |\nabla^{3}\varphi_{n}(Y_{n}+rH_{n})| \,|H_{n}|^{2} \diff{r}
            \,\Big|\, \F_{t_n}\bigg)
            \\\leq&~
            C\Delta |f(t_{n},Y_{n})| \,\E\big( 
            (1 + |Y_{n}|^{q} + |H_{n}|^{q})|H_{n}|^{2}
            \mid \F_{t_n}\big)
            \\\leq&~
            C\big(1 + |Y_{n}|^{q}\big)\Delta^{2},
      \end{align*}
      where the polynomial growth of $\nabla^{3}\varphi_{n}$, 
      the linear growth of $f$, and \eqref{eq:Hnpestimate} have been used.
      In terms of $\frac{\Delta}{2}\E\big(\tr\big(g(t_{n},Y_{n})^{\top}
      \nabla^{2}\varphi_{n}(Y_{n}+H_{n})g(t_{n},Y_{n})\big) \mid \F_{t_n}\big)$,
      the Taylor expansion 
      \begin{align*}
            \nabla^{2}\varphi_{n}(Y_{n}+H_{n}) - \nabla^{2}\varphi_{n}(Y_{n})
            =
            \nabla^{3}\varphi_{n}(Y_{n})H_{n}
            +
            \int_{0}^{1} (1-r)\nabla^{4}\varphi_{n}(Y_{n}+rH_{n}) H_{n}H_{n} \diff{r}
      \end{align*}
      and  $\E(H_n \mid \F_{t_n}) = 0$ show that
      \begin{align}\label{eq:Trgthn2vargtn}
            &~\frac{\Delta}{2}\E\big(\tr\big(g(t_{n},Y_{n})^{\top}
            \nabla^{2}\varphi_{n}(Y_{n}+H_{n})g(t_{n},Y_{n})\big) \mid \F_{t_n}\big) \notag
            \\=&~
            \frac{\Delta}{2}\tr\big(g(t_{n},Y_{n})^{\top}
            \nabla^{2}\varphi_{n}(Y_{n}) g(t_{n},Y_{n})\big)  \notag
            \\&~+
            \frac{\Delta}{2}\E\bigg(\tr\bigg(g(t_{n},Y_{n})^{\top}
            \Big(\int_{0}^{1} (1-r)\nabla^{4}\varphi_{n}(Y_{n}+rH_{n}) 
            H_{n}H_{n} \diff{r}\big)
            g(t_{n},Y_{n})\bigg) \,\Big|\, \F_{t_n}\bigg)   \notag
            \\=&~
            \frac{\Delta}{2}\tr\big(g(t_{n},Y_{n})^{\top}
            \nabla^{2}\varphi_{n}(Y_{n}) g(t_{n},Y_{n})\big)
            +
            \mathcal{O}\big(\big(1 + |Y_{n}|^{q}\big)\Delta^{2}\big).         
      \end{align}
      Applying \eqref{eq:chinpestimate}, \eqref{eq:Rnbterms}, 
      \eqref{eq:Hnpestimate}, the polynomial growth of $\nabla^{3}\varphi_{n}$, 
      and the H\"{o}lder inequality yields
      \begin{align}\label{eq:ERnBFtn}
            \big|\E\big(R_{n}^{B} \mid \F_{t_n}\big)\big| \notag
            \leq&~
            \bigg|\E\bigg(\int_{0}^{1} \frac{(1-r)^{3}}{6}
            |\nabla^{4}\varphi_{n}(Y_{n}+H_{n} + r\chi_{n})|
            |\chi_{n}|^{4}\diff{r} 
            \,\Big|\, \F_{t_n}\bigg)\bigg| \notag
            \\\leq&~
            C\E\big((1 + |Y_{n}|^{q} + |H_{n}|^{q} + |\chi_{n}|^{q})
            |\chi_{n}|^{4} \mid \F_{t_n}\big) \notag
            \\\leq&~
            C(1 + |Y_{n}|^{q})\E\big(
            |\chi_{n}|^{4} \mid \F_{t_n}\big)
            +
            C\E\big(|H_{n}|^{q}
            |\chi_{n}|^{4} \mid \F_{t_n}\big)
            +
            C\E\big(|\chi_{n}|^{q+4} \mid \F_{t_n}\big) \notag
            \\\leq&~
            C(1 + |Y_{n}|^{q})\Delta^{2}
            +
            C\big(\E\big(|H_{n}|^{2q}
            \mid \F_{t_n}\big)\big)^{\frac{1}{2}}
            \big(\E\big(|\chi_{n}|^{8} 
            \mid \F_{t_n}\big)\big)^{\frac{1}{2}} \notag
            \\\leq&~
            C(1 + |Y_{n}|^{q})\Delta^{2}.
      \end{align}
      Putting \eqref{eq:EvarphinYnplusHn}, \eqref{eq:DelatvarphinYnHn}, 
      \eqref{eq:Trgthn2vargtn} and \eqref{eq:ERnBFtn} into \eqref{eq:varphibarYnftn} results in 
      \begin{align}\label{eq:varphibarYnftnoutcome}
            \E\big(\varphi_{n}(\overline{Y}_{n+1}) \mid \F_{t_n}\big) \notag
            =&~
            \varphi_{n}(Y_{n}) 
            + 
            \Delta \mathcal{L}_{t_{n}}\varphi_{n}(Y_{n})
            +
            \mathcal{O}\big(\big(1 + |Y_{n}|^{q}\big)\Delta^{2}\big) \notag
            \\=&~
            u(t_{n+1},Y_{n})
            + 
            \Delta \mathcal{L}_{t_{n}}u(t_{n+1},Y_{n})
            +
            \mathcal{O}\big(\big(1 + |Y_{n}|^{q}\big)\Delta^{2}\big),
      \end{align}
      where \eqref{eq:mathcalLs} and $\varphi_{n}(x) := u(t_{n+1},x)$ have been used.

      To proceed, integrating $\partial_{s}u(s,Y_{n}) + \mathcal{L}_{s}u(s,Y_{n}) = 0$
      for $s \in [t_{n},t_{n+1}]$ gives
      \begin{align*}
            &~u(t_{n+1},Y_{n}) 
             =
            u(t_{n},Y_{n})            
            - 
            \Delta \mathcal{L}_{t_{n}}u(t_{n+1},Y_{n})
            - 
            \int_{t_n}^{t_{n+1}} \mathcal{L}_{s}u(s,Y_{n})
            - \mathcal{L}_{t_n}u(t_{n+1},Y_{n}) \diff{s}.
      \end{align*}
      Owing to Lemmas \ref{lem:mathcalLslip1} and \ref{lem:mathcalLslip2}, we deduce that
      \begin{align*}
            \bigg|\int_{t_n}^{t_{n+1}} \mathcal{L}_{s}u(s,Y_{n})
            - \mathcal{L}_{t_n}u(t_{n+1},Y_{n}) \diff{s}\bigg|
            \leq
            C\int_{t_n}^{t_{n+1}}(1+|Y_{n}|^{q})\Delta \diff{s}
            \leq
            C(1 + |Y_{n}|^{q})\Delta^{2},
      \end{align*}
      and thus $u(t_{n+1},Y_{n}) + \Delta \mathcal{L}_{t_{n}}u(t_{n+1},Y_{n})
      = u(t_{n},Y_{n}) + \mathcal{O}\big(\big(1 + |Y_{n}|^{q}\big)\Delta^{2}\big)$.
      In combination with \eqref{eq:varphibarYnftnoutcome}, we derive the desired 
      conclusion and thus complete the proof.
\end{proof}

\subsection{Weak contribution of the implicit correction}  
This subsection aims to evaluate the weak contribution of the implicit correction $\delta_{n+1} := Y_{n+1} - \overline{Y}_{n+1}$ introduced by the stochastic theta method \eqref{eq:Stmethod} relative to the purely explicit frozen Euler step \eqref{eq:YEMfrozen}. To this end, we first provide the conditional moment estimates for the relevant increments.

\begin{lemma}
      Suppose that Assumptions \ref{ass:globallinear}, \ref{As:fghgrowthcond} hold, and let $\theta{L}\Delta \leq \frac{1}{2}$ with $\theta \in [0,1]$. Then for any $p \geq 2$, there exists a constant $C > 0$ and an integer $q \geq 1$ such that for any $n \in \N$,
      \begin{align}\label{eq:barYestimate}
            \E\big(|\overline{Y}_{n+1} - Y_{n}|^{p} \mid \F_{t_{n}}\big)
            \leq
            C\Delta(1 + |Y_{n}|^{q}),
            \quad
            \E\big(|\overline{Y}_{n+1}|^{p} \mid \F_{t_{n}}\big)
            \leq
            C(1 + |Y_{n}|^{q}).            
      \end{align}
\end{lemma}

\begin{proof}
      Since $Y_{n}$ is $\F_{t_n}$-measurable, we have
      \begin{align*}
            \E\big(|\overline{Y}_{n+1} - Y_{n}|^{p} \mid \F_{t_{n}}\big)
            \leq&~
            3^{p-1}\E\big(|f(t_{n},Y_{n})\Delta|^{p} \mid \F_{t_{n}}\big)
            +
            3^{p-1}\E\big(|g(t_{n},Y_{n})\Delta W_{n} |^{p} \mid \F_{t_{n}}\big)
            \\&~+
            3^{p-1}\E\bigg(\bigg|\int_{t_{n}}^{t_{n+1}}\int_{|z|<c}
            h(t_n,Y_n,z) \,\widetilde{N}(\dif{z},\dif{s})\bigg|^{p} \,\Big|\, \F_{t_{n}}\bigg)
            \\\leq&~
            C|f(t_{n},Y_{n})\Delta|^{p}
            +
            C|g(t_{n},Y_{n})|^{p}
            \E\big[|\Delta W_{n} |^{p}\big]
            \\&~+
            C\E\bigg[\bigg(\int_{t_{n}}^{t_{n+1}}\int_{|z|<c}
            |h(t_n,Y_n,z)|^{2} \,\mu(\dif{z})\dif{s}\bigg)^{\frac{p}{2}}\bigg]
            \\&~+
            C\E\bigg[\int_{t_{n}}^{t_{n+1}}\int_{|z|<c}
            |h(t_n,Y_n,z)|^{p} \,\mu(\dif{z})\dif{s}\bigg]
            \\\leq&~
            C\Delta(1 + |Y_{n}|^{q})
      \end{align*}
      due to the linear growth of $f,g,h$. Together with  
      \begin{align*}
             \E\big(|\overline{Y}_{n+1}|^{p} \mid \F_{t_{n}}\big)
             \leq&~
             2^{p-1}\E\big(|\overline{Y}_{n+1} - Y_{n}|^{p} \mid \F_{t_{n}}\big)
             +
             2^{p-1}|Y_{n}|^{p},
      \end{align*}
      we complete the proof.
\end{proof}

Building upon the increment bounds, we now establish the conditional moment estimates for the implicit correction term $\delta_{n+1}$ itself.
\begin{lemma}
      Suppose that Assumptions \ref{ass:globallinear}, \ref{As:fghgrowthcond}, \ref{ass:fghtimeLipschitz} hold, and let $\theta{L}\Delta \leq \frac{1}{2}$ with $\theta \in [0,1]$. Then for any $p \geq 2$, there exists a constant $C > 0$ and an integer $q \geq 1$ such that for any $n \in \N$,
      \begin{equation}\label{eq:delta-mean-var}
            \big|\E(\delta_{n+1}\mid\F_{t_n})\big|\le C(1+|Y_n|^q)\Delta^2,
            \quad
            \E(|\delta_{n+1}|^{p}\mid\F_{t_n})\le C(1+|Y_n|^q)\Delta^{p+1}.
      \end{equation}
\end{lemma}

\begin{proof}
      By \eqref{eq:Stmethod} and \eqref{eq:YEMfrozen}, we have
      \begin{align}\label{eq:deltan1f}
            \delta_{n+1} = \theta\big(f(t_{n+1},Y_{n+1}) - f(t_{n},Y_{n})\big)\Delta. 
      \end{align}      
      Letting $\eta_{n+1} := Y_{n+1}-Y_{n}$ and using the first-order Taylor formula with integral remainder for $x \mapsto f(t_{n+1},x)$ at $Y_n$, we have $f(t_{n+1},Y_{n+1}) = f(t_{n+1},Y_n) + \nabla f(t_{n+1},Y_n)\eta_{n+1} + R_{n}^{f}$ with
      $$R_{n}^{f} := \int_{0}^{1} (1-r) \nabla^{2} f\big(t_{n+1},
      Y_n+r\eta_{n+1}\big)[\eta_{n+1},\eta_{n+1}] \diff{r}.$$
      It follows from $\eta_{n+1} = \delta_{n+1} + \overline{Y}_{n+1}-Y_n$ and \eqref{eq:deltan1f} that
      \begin{align*}
            \big(I-\theta\Delta (\nabla f(t_{n+1},Y_{n}))\big)\delta_{n+1}
            =&~
            \theta\Delta\big(f(t_{n+1},Y_n)-f(t_n,Y_n)\big)
            \\&~+
            \theta\Delta (\nabla f(t_{n+1},Y_n))\big(\overline{Y}_{n+1}-Y_n\big)
            + \theta\Delta R_{n}^{f} .
      \end{align*}
      Since $f$ is globally Lipschitz in $x$ with Lipschitz constant $L$, we have $|\nabla f(t_{n+1},Y_n))| \leq L$. Together with $\theta\Delta L \leq 1/2$, we deduce that $I-\theta\Delta (\nabla f(t_{n+1},Y_n))$ is invertible and
      \begin{equation*}
            \big|\big(I-\theta\Delta (\nabla f(t_{n+1},Y_{n}))\big)^{-1}\big|
            \leq \frac1{1-\theta\Delta L} \le 2 .
      \end{equation*}
      Applying the $\F_{t_{n}}$-measurability of $\big(I-\theta\Delta (\nabla f(t_{n+1},Y_{n}))\big)^{-1}$ yields
      \begin{align*}
            \big|\E\big(\delta_{n+1} \mid \F_{t_n}\big)\big|
            =&~
            \big|\big(I-\theta\Delta (\nabla f(t_{n+1},Y_{n}))\big)^{-1} 
            \E\big(\big(\theta\Delta\big(f(t_{n+1},Y_n)-f(t_n,Y_n)\big)
            \\&~+
            \theta\Delta (\nabla f(t_{n+1},Y_n))\big(\overline{Y}_{n+1}-Y_n\big)
            + \theta\Delta R_{n}^{f}\big)
             \mid \F_{t_n}\big)\big|
            \\\leq&~
            \big|\big(I-\theta\Delta (\nabla f(t_{n+1},Y_{n}))\big)^{-1}\big|
            \big|\theta\Delta\big(f(t_{n+1},Y_n)-f(t_n,Y_n)\big)
            \\&~+
            \theta\Delta (\nabla f(t_{n+1},Y_n))
            \E\big( \overline{Y}_{n+1}-Y_n \mid \F_{t_n}\big)
            + \theta\Delta \E\big(R_{n}^{f} \mid \F_{t_n}\big)\big|
            \\\leq&~
            2\theta\Delta\big(\big|f(t_{n+1},Y_n)-f(t_n,Y_n)\big|
            + 
            \big|\E\big(R_{n}^{f} \mid \F_{t_n}\big)\big|
            \\&~+
            \big|\nabla f(t_{n+1},Y_n)\big|
            \big|\E\big( \overline{Y}_{n+1}-Y_n \mid \F_{t_n}\big)\big|\big)
      \end{align*}
      Owing to $\E(\Delta W_n \mid \F_{t_n})=0$ and $\E(H_{n} \mid \F_{t_n})=0$, we obtain $\E(\overline{Y}_{n+1}-Y_n\mid\F_{t_n})=f(t_n,Y_n)\Delta$. 
      Together with the time-Lipschitz continuity of $f$, i.e., $|f(t_{n+1},Y_n) 
      - f(t_n,Y_n)| \leq C(1+|Y_n|^q)\Delta$, we obtain
      \begin{equation}\label{eq:delta-mean-pre}
            \big|\E\big(\delta_{n+1} \mid \F_{t_n}\big)\big|
            \leq
             C(1+|Y_{n}|^{q})\Delta^{2} 
             + 
             C\Delta \E\big(|R_{n}^{f}| \mid \F_{t_n}\big).
      \end{equation}
      It remains to bound $\E(|R_n^{f}|\mid\F_{t_n})$. Noting that 
      \begin{align*}
            \eta_{n+1} 
            =&~
            f(t_{n},Y_{n})\Delta
            +
            g(t_{n},Y_{n})\Delta W_{n}  \notag
            +
            \int_{t_{n}}^{t_{n+1}}\int_{|z|<c}
            h(t_n,Y_n,z) \,\widetilde{N}(\dif{z},\dif{s})
            \\&~+
            \theta \big(f(t_{n+1},Y_{n})
            -  f(t_{n},Y_{n})\big) \Delta 
            +
            \theta \big(f(t_{n+1},Y_{n+1})
            -  f(t_{n+1},Y_{n})\big) \Delta,
      \end{align*}
      we apply the space-Lipschitz continuity of $f$ to get
      \begin{align*}
	        |\eta_{n+1}| 
            \leq&~
            \bigg|f(t_{n},Y_{n})\Delta
            +
            g(t_{n},Y_{n})\Delta W_{n}  \notag
            +
            \int_{t_{n}}^{t_{n+1}}\int_{|z|<c}
            h(t_n,Y_n,z) \,\widetilde{N}(\dif{z},\dif{s})
            \\&~+
	        \theta \big(f(t_{n+1},Y_{n})
            -  f(t_{n},Y_{n})\big) \Delta \bigg|
            +
	        L\theta\Delta |\eta_{n+1}|.
      \end{align*}
      Then $\theta\Delta L \leq 1/2$ and $\frac1{1-\theta\Delta L} \le 2$ ensure that
      \begin{align*}
	        |\eta_{n+1}| 
	        \leq&~
            C\bigg|f(t_{n},Y_{n})\Delta
            +
            g(t_{n},Y_{n})\Delta W_{n}  \notag
            +
            \int_{t_{n}}^{t_{n+1}}\int_{|z|<c}
            h(t_n,Y_n,z) \,\widetilde{N}(\dif{z},\dif{s})
            \\&~+
	        \theta \big(f(t_{n+1},Y_{n})
            -  f(t_{n},Y_{n})\big) \Delta \bigg|,
      \end{align*}
      and accordingly that for any $p \geq 2$,
      \begin{align}\label{eq:etan1pFtn}
            \E(|\eta_{n+1}|^{p} \mid \F_{t_n})
            \leq 
            C(1+|Y_n|^{q})\Delta.
      \end{align}
      Since $\nabla^{2} f$ has polynomial growth, there exists an integer $q \geq 1$ such that
      \begin{align*}
            |R_{n}^{f}|
            \leq&~
            \int_{0}^{1} (1-r) |\nabla^{2} f\big(t_{n+1},
            Y_n+r\eta_{n+1}\big)[\eta_{n+1},\eta_{n+1}]| \diff{r}
            \\\leq&~
            C\big(1+|Y_n|^q+|\eta_{n+1}|^q\big)|\eta_{n+1}|^{2},
      \end{align*}
      which in combination with \eqref{eq:etan1pFtn} gives
      \begin{align}\label{eq:Rnfcond}
            \E\big(|R_{n}^{f}| \mid \F_{t_n}\big)
            \leq 
            C(1+|Y_n|^q)\E\big(|\eta_{n+1}|^2 \mid \F_{t_n}\big)
            + C\E\big(|\eta_{n+1}|^{q+2} \mid \F_{t_n}\big)
            \leq
            C(1+|Y_n|^q)\Delta.
      \end{align}
      Substituting \eqref{eq:Rnfcond} into \eqref{eq:delta-mean-pre} yields the first estimate in \eqref{eq:delta-mean-var}.
      From \eqref{eq:deltan1f} and the space/time Lipschitz continuity of $f$, it holds that
      \begin{align*}
            |\delta_{n+1}| 
            \leq&~
            \theta\Delta\big(|f(t_{n+1},Y_{n+1}) - f(t_{n+1},Y_{n})|
            + |f(t_{n+1},Y_{n}) - f(t_{n},Y_{n})|\big)
            \\\leq&~
            C\Delta\big(|\eta_{n+1}|+(1+|Y_n|^q)\Delta\big),
      \end{align*}
      and hence $|\delta_{n+1}|^{p} \leq C\Delta^{p}|\eta_{n+1}|^{p} + C(1+|Y_n|^{pq})\Delta^{2p}$ for any $p \geq 2$. 
      Taking $\E(\cdot\mid\F_{t_n})$ and using \eqref{eq:etan1pFtn}, 
      we obtain the second estimate in \eqref{eq:delta-mean-var}.  
\end{proof}

Leveraging the rigorous moment control over the implicit correction term $\delta_{n+1}$, we can now bound its higher-order weak error contribution within a single-step conditional expectation.
\begin{proposition}\label{prop:STMminusEM}
      Suppose that Assumptions \ref{ass:globallinear}, \ref{As:fghgrowthcond}, \ref{ass:fghtimeLipschitz} hold, and let $\theta{L}\Delta \leq \frac{1}{2}$ with $\theta \in [0,1]$. Then there exists a constant $C > 0$ and an integer $q \geq 1$ such that for any $n \in \N$,
      \begin{align*}
            \big|\E\big( u(t_{n+1},Y_{n+1})
            - u(t_{n+1},\overline{Y}_{n+1}) \mid \F_{t_{n}}\big)\big|
            \leq 
            C\big(1 + |Y_{n}|^{q}\big)\Delta^{2}.
      \end{align*}
\end{proposition}

\begin{proof}
      Since $x \mapsto u(t_{n+1},x) =: \varphi_{n}(x)$ belongs to $C_{p}^{4}(\R^{d};\R)$, using the Taylor formula twice yields
      \begin{align*}
            &~u(t_{n+1},Y_{n+1})
            - u(t_{n+1},\overline{Y}_{n+1})
            \\=&~
            \big\langle \nabla\varphi_{n}(\overline{Y}_{n+1}), \delta_{n+1}\big\rangle
            +
            \frac{1}{2} \delta_{n+1}^{\top}
            \nabla^{2}\varphi_{n}(\overline{Y}_{n+1})\delta_{n+1} 
            + R_{n+1}
            \\=&~
            \big\langle \nabla\varphi_{n}(Y_{n}), \delta_{n+1}\big\rangle
            +
            \big\langle \nabla^{2}\varphi_{n}(\xi_{n+1})(\overline{Y}_{n+1}-Y_{n}), 
            \delta_{n+1}\big\rangle
            +
            \frac{1}{2} \delta_{n+1}^{\top}
            \nabla^{2}\varphi_{n}(\overline{Y}_{n+1})\delta_{n+1} 
            + R_{n+1}
      \end{align*}
      with $\xi_{n+1} :=  Y_{n} + \kappa (\overline{Y}_{n+1}-Y_{n})$ for some $\kappa \in [0,1]$ and
      $$R_{n+1} = \frac{1}{2} \int_{0}^{1} (1-r)
        \nabla^{3}\varphi_{n}(\overline{Y}_{n+1} + r\delta_{n+1})
        [\delta_{n+1},\delta_{n+1},\delta_{n+1}] \diff{r}.$$
      It follows that
      \begin{align}\label{eq:EutYbarY}
            &~\big|\E\big(u(t_{n+1},Y_{n+1})
            - u(t_{n+1},\overline{Y}_{n+1})
            \,|\, \F_{t_{n}}\big)\big| \notag
            \\\leq&~
            \big|\E\big(\big\langle \nabla\varphi_{n}(Y_{n}), 
            \delta_{n+1}\big\rangle  \mid \F_{t_{n}}\big)\big|
            + 
            \big|\E\big(\big\langle \nabla^{2}\varphi_{n}(\xi_{n+1})
            (\overline{Y}_{n+1}-Y_{n}), 
            \delta_{n+1}\big\rangle  \mid \F_{t_{n}}\big)\big|  \notag
            \\&~+
            \frac{1}{2} \big|\E\big( \delta_{n+1}^{\top}
            \nabla^{2}\varphi_{n}(\overline{Y}_{n+1})\delta_{n+1} 
             \,|\, \F_{t_{n}}\big)\big|
            +
            \big|\E\big(R_{n+1} \mid \F_{t_{n}}\big)\big|.
      \end{align}    
      For $\big|\E\big(\big\langle \nabla\varphi_{n}(Y_{n}), 
            \delta_{n+1}\big\rangle  \mid \F_{t_{n}}\big)\big|$, 
      we utilize \eqref{eq:delta-mean-var} and the polynomial growth of $\nabla\varphi_{n}$ to derive
      \begin{align}\label{eq:nablavarnYn}
            \big|\E\big(\big\langle \nabla\varphi_{n}(Y_{n}), 
            \delta_{n+1}\big\rangle  \mid \F_{t_{n}}\big)\big| 
            \leq
            \big|\nabla\varphi_{n}(Y_{n}) \big| 
            \big|\E\big(\delta_{n+1}  \mid \F_{t_{n}} \big)\big|
            \leq
            C(1+|Y_n|^q)\Delta^2.
      \end{align}
      Concerning $\big|\E\big(\big\langle \nabla^{2}\varphi_{n}(\xi_{n+1})
            (\overline{Y}_{n+1}-Y_{n}), 
            \delta_{n+1}\big\rangle  \mid \F_{t_{n}}\big)\big|$,
      application of the polynomial growth of $\nabla^{2}\varphi_{n}$, 
      the H\"{o}lder inequality, \eqref{eq:barYestimate}, 
      and \eqref{eq:delta-mean-var} enables us to get
      \begin{align}\label{eq:babla2varnxin}
            &~\big|\E\big(\big\langle \nabla^{2}\varphi_{n}(\xi_{n+1})
            (\overline{Y}_{n+1}-Y_{n}), 
            \delta_{n+1}\big\rangle  \mid \F_{t_{n}}\big)\big| \notag
            \\\leq&~
            \E\big(|\nabla^{2}\varphi_{n}(\xi_{n+1})|\,
            |\overline{Y}_{n+1} - Y_{n}|\,
            |\delta_{n+1}|  \mid \F_{t_{n}}\big) \notag
            \\\leq&~
            C\E\big((1 + |Y_{n}|^{q} + |\overline{Y}_{n+1}-Y_{n}|^{q})\,
            |\overline{Y}_{n+1} - Y_{n}|\,
            |\delta_{n+1}|  \mid \F_{t_{n}}\big) \notag
            \\\leq&~
            C\big(\E\big((1 + |Y_{n}|^{q} + |\overline{Y}_{n+1}-Y_{n}|^{q})^{2}\,
            |\overline{Y}_{n+1} - Y_{n}|^{2}
            \mid \F_{t_{n}}\big)\big)^{\frac{1}{2}}
            \big(\E\big(|\delta_{n+1}|^{2}  
            \mid \F_{t_{n}}\big)\big)^{\frac{1}{2}} \notag
            \\\leq&~
            C(1+|Y_n|^{q})\Delta^{\frac{3}{2}}
            \big((1 + |Y_{n}|^{2q})
            \E\big(|\overline{Y}_{n+1} - Y_{n}|^{2}
            \mid \F_{t_{n}}\big)
            +
            \E\big(|\overline{Y}_{n+1}-Y_{n}|^{2q+2}
            \mid \F_{t_{n}}\big)\big)^{\frac{1}{2}} \notag
            \\\leq&~
            C(1+|Y_n|^{q})\Delta^{2}.
      \end{align}
      As to $\frac{1}{2} \big|\E\big( \delta_{n+1}^{\top}
      \nabla^{2}\varphi_{n}(\overline{Y}_{n+1})\delta_{n+1} \mid \F_{t_{n}}\big)\big|$,
      one can use the polynomial growth of $\nabla^{2}\varphi_{n}$, \eqref{eq:barYestimate}, and 
      \eqref{eq:delta-mean-var} to deduce that
      \begin{align}\label{eq:deltanbarY}
            &~\frac{1}{2} \big|\E\big( \delta_{n+1}^{\top}\nabla^{2}
            \varphi_{n}(\overline{Y}_{n+1})\delta_{n+1} \mid \F_{t_{n}}\big)\big| \notag
            \\\leq&~
            C \E\big( |\delta_{n+1}|^{2} 
            |\nabla^{2} \varphi_{n}(\overline{Y}_{n+1})| \mid \F_{t_{n}}\big)  \notag
            \\\leq&~
            C \big(\E\big( |\delta_{n+1}|^{4} 
            \mid \F_{t_{n}}\big)\big)^\frac{1}{2}
            \big(\E\big( (1 + |\overline{Y}_{n+1}|^{q})^{2}
            \mid \F_{t_{n}}\big)\big)^\frac{1}{2} \notag
            \\\leq&~
            C\Delta^{\frac{5}{2}}(1 + |Y_{n}|^{q}).
      \end{align}
      For $\big|\E\big(R_{n+1} \mid \F_{t_{n}}\big)\big|$, 
      employing the polynomial growth of $\nabla^{3}\varphi_{n}$,
      \eqref{eq:barYestimate}, and \eqref{eq:delta-mean-var} leads to
      \begin{align}\label{eq:ERFtn}
            \big|\E\big(R_{n+1} \mid \F_{t_{n}}\big)\big|
            \leq&~
            \E\bigg(\frac{1}{2} \int_{0}^{1} (1-r)
            |\nabla^{3}\varphi_{n}(\overline{Y}_{n+1} + r\delta_{n+1})|
            |\delta_{n+1}|^{3} \diff{r} \mid \F_{t_{n}}\bigg) \notag
            \\\leq&~
            C\E\big((1 + |\overline{Y}_{n+1}|^{q} + |\delta_{n+1})|^{q})
            |\delta_{n+1}|^{3} \mid \F_{t_{n}}\big) \notag
            \\\leq&~
            C\big(\E\big((1 + |\overline{Y}_{n+1}|^{q})^{2}
            \mid \F_{t_{n}}\big)\big)^{\frac{1}{2}}
            \big(\E\big(|\delta_{n+1}|^{6} 
            \mid \F_{t_{n}}\big)\big)^{\frac{1}{2}} \notag
            +
            C\E\big(|\delta_{n+1}|^{q+3} \mid \F_{t_{n}}\big) \notag
            \\\leq&~
            C\Delta^{\frac{7}{2}}(1 + |Y_{n}|^{q}).
      \end{align}
      Putting \eqref{eq:nablavarnYn}, \eqref{eq:babla2varnxin}, 
      \eqref{eq:deltanbarY}, and \eqref{eq:ERFtn} into \eqref{eq:EutYbarY}
       finishes the proof. 
\end{proof}

\subsection{Weak convergence order of stochastic theta method}
With the aid of the moment bounds established in Proposition \ref{prop:EYnpestimate} to control the polynomial growth of the remainder terms, we now combine the weak consistency of the frozen Euler step in Proposition \ref{prop:EMerror} and the higher-order weak contribution of the implicit correction in Proposition \ref{prop:STMminusEM} via a global telescoping argument. This allows us to establish the overall weak convergence order of the stochastic theta method for the non-time-changed SDE.

\begin{proposition}\label{prop:levyweakorder}
      Assume that all components of $f(t,\cdot), g(t,\cdot), h(t,\cdot,z)$ belong to $C_{p}^{4}(\R^{d};\R)$ for any $t \geq 0$ and $z \in \R^{d} \backslash \{0\}$. Suppose that Assumptions \ref{ass:globallinear}, \ref{As:fghgrowthcond}, \ref{ass:fghtimeLipschitz} hold, and let $\theta{L}\Delta \leq \frac{1}{2}$ with $\theta \in [0,1]$. Then there exists a constant $C > 0$ such that for any $N \in \N$ and $T = N\Delta$,
      \begin{align}
            \big|\E\big[\Phi(Y_{N})\big] - \E\big[\Phi(Y(T))\big]\big| 
            \leq 
            C\Delta e^{CT}.
      \end{align}
\end{proposition}

\begin{proof}
      Applying the It\^{o} formula and using \eqref{eq:SDE} and \eqref{eq:Kbpide}, one has 
      \begin{align*}
			\diff{u(s,Y(s))}
			=&~
			\big\langle (\nabla u(s,Y(s))),
            g(s,Y(s)) \diff{W(s)} \big\rangle
			\\&~+ 
			\int_{|z|<c}u\big(s,Y(s)+h(s,Y(s),z)\big)
			- u(s,Y(s))\,\widetilde{N}(\dif{z},\dif{s}),
            \quad s \in (0,T],
      \end{align*}
      which implies $\E[u(0,Y(0))] = \E[u(s,Y(s))] = \E[\Phi(Y(s))]$ for $s \in (0,T]$. It follows from $u(t_{N},Y_{N}) = \Phi(Y_{N})$ and $u(0,Y(0)) = \E\big[\Phi(Y(T))\big]$ that
      \begin{align*}
            \E\big[\Phi(Y_{N})\big] - \E\big[\Phi(Y(T))\big] 
            =&~
            \sum_{n=0}^{N-1} \E\big[u(t_{n+1},Y_{n+1})\big]
            - \E\big[u(t_{n},Y_{n})\big]
            \\=&~
            \sum_{n=0}^{N-1}\Big(
            \E\big[ \E\big( u(t_{n+1},Y_{n+1})
            - u(t_{n+1},\overline{Y}_{n+1}) \,|\, \F_{t_{n}}\big) \big] 
            \\&~+
            \E\big[ \E\big( u(t_{n+1},\overline{Y}_{n+1})
            - u(t_{n},Y_{n}) \,|\, \F_{t_{n}}\big) \big]\Big).          
      \end{align*} 
      Propositions \ref{prop:EYnpestimate}, \ref{prop:STMminusEM}, and \ref{prop:EMerror} indicate 
      \begin{align*}
            \big|\E\big[\Phi(Y_{N})\big] - \E\big[\Phi(Y(T))\big]\big| 
            \leq&~
            \sum_{n=0}^{N-1}\Big(
            \E\big[\big|\E\big( u(t_{n+1},Y_{n+1})
            - u(t_{n+1},\overline{Y}_{n+1}) \,|\, \F_{t_{n}}\big)\big|\big] 
            \\&~+
            \E\big[\big|\E\big( u(t_{n+1},\overline{Y}_{n+1})
            - u(t_{n},Y_{n}) \,|\, \F_{t_{n}}\big)\big|\big]\Big) 
            \\\leq&~
            C\Delta^{2}\sum_{n=0}^{N-1}
            \big(1 + \E\big[|Y_{n}|^{q}\big]\big)  
            \\\leq&~
            C\Delta^{2}\sum_{n=0}^{N-1}
            \big(1 + Ce^{Ct_{n}}\big) 
            \\\leq&~
            CN\Delta^{2}
            \big(1 + Ce^{CT}\big)     
            \\\leq&~
            C\Delta e^{CT},   
      \end{align*}
      and thus complete the proof.
\end{proof}

Now we turn to approximate $\{E(t)\}_{t \in [0,T]}$ via $\{E_{\Delta}(t)\}_{t \in [0,T]}$, which is defined by
\begin{equation}\label{eq:EDeltatdef}
      E_{\Delta}(t)
      :=
      (\min\{n \in \N : D(t_n)>t\} - 1)\Delta,
      \quad t \in [0,T].
\end{equation}
Then each sample path of stochastic process $\{E_{\Delta}(t)\}_{t \in [0,T]}$ is a non-decreasing step function with constant jump size $\Delta$, and the procedure is stopped when $T \in [D(t_{N}), D(t_{N+1}))$ for some $N \in \N$. It follows that $E_\Delta(T) = t_{N}$ and 
\begin{equation}\label{eq:EDeltattn}
      E_{\Delta}(t)
      =
      t_{n}, \quad t \in [D(t_{n}),D(t_{n+1})),
      n = 0, 1,\cdots, N.
\end{equation}
Moreover, Lemma 3.5 in \cite{jum2016strong} ensures that $\{E(t)\}_{t \in [0,T]}$ can be well approximated by $\{E_{\Delta}(t)\}_{t \in [0,T]}$ in the sense that 
\begin{equation}\label{eq:EtEDeltaterror}
      E(t)-\Delta 
      \leq
      E_{\Delta}(t)\leq E(t), 
      \quad t \in [0,T].
\end{equation}
Recalling the fact that the solution $\{X(t)\}_{t \in [0,T]}$ to \eqref{eq:TCSDE} can be expressed as $X(t) = Y(E(t)), t \in [0,T]$ with $\{Y(t)\}_{t \geq 0}$ denoting the solution to \eqref{eq:SDE}  by Lemma \ref{lem:dualityprinciple}. Hence, it is reasonable to  approximate $\{X(t)\}_{t \in [0,T]}$ by $\{X_{\Delta}(t)\}_{t \in [0,T]}$ as follows
\begin{equation}\label{eq:XDeltaYDelta}
      X_{\Delta}(t)
      :=
      Y_{\Delta}(E_\Delta(t)), \quad t \in [0,T],
\end{equation}
where the piecewise continuous process $\{Y_{\Delta}(t)\}_{t \geq 0}$ is defined by
\begin{equation}\label{eq:YDeltaYn}
      Y_{\Delta}(t) = Y_{n}, 
      \quad t \in [t_{n},t_{n+1}), n = 0,1,2,\cdots.
\end{equation}

Finally, by integrating the weak convergence result of the non-time-changed equation with the numerical approximation error of inverse subordinator, we arrive at the overall weak convergence order of the numerical approximation process $\{X_{\Delta}(t)\}_{t \in [0,T]}$ for the time-changed SDE.
				
\begin{theorem}\label{th:weakorder}
      Assume that all components of $f(t,\cdot), g(t,\cdot), h(t,\cdot,z)$ belong to $C_{p}^{4}(\R^{d};\R)$ for any $t \geq 0$ and $z \in \R^{d} \backslash \{0\}$. Suppose that Assumptions \ref{ass:globallinear}, 
      \ref{ass:techniquecond}, \ref{As:fghgrowthcond}, \ref{ass:fghtimeLipschitz} hold, and let $\theta{L}\Delta \leq \frac{1}{2}$ with $\theta \in [0,1]$. Then there exists $C > 0$, independent of $\Delta$, such that for any $\Phi \in C_{p}^{4}(\R^{d};\R)$,
      \begin{align*}
			\big|\E\big[\Phi(X(T))\big]
            - \E\big[\Phi(X_\Delta(T))\big]\big|
			\leq
			C\Delta.    
      \end{align*}
\end{theorem}

\begin{proof}
      By \eqref{eq:XDeltaYDelta} and Lemma \ref{lem:dualityprinciple}, we have
      \begin{align*}
			&~\big|\E\big[\Phi(X(T))\big]
		    - \E\big[\Phi(X_\Delta(T))\big]\big|
			=
			\big|\E\big[\Phi(Y(E(T))
			- \Phi(Y_\Delta(E_\Delta (T))\big]\big|
			\\\leq&~
			\big|\E\big[\Phi(Y(E(T)) - \Phi(Y(E_\Delta (T))\big]\big|
			+
			\big|\E\big[\Phi(Y(E_\Delta (T)) - \Phi(Y_\Delta(E_\Delta (T))\big]\big|.    
      \end{align*}
      Applying \eqref{eq:estimateEelambdaEt}, \eqref{eq:EtEDeltaterror} and Lemma \ref{lm:Ytpexactbound} leads to
      \begin{align*}
			\big|\E\big[\Phi(Y(E(T)) - \Phi(Y(E_\Delta (T))\big]\big|
			\leq
			\Delta C\E\big[e^{CE(T)}\big]
			\leq 
			\Delta Ce^{CT}.
      \end{align*}
      Owing to $E_\Delta(T) = t_{N}$, \eqref{eq:estimateEelambdaEt}, \eqref{eq:EtEDeltaterror} and Proposition \ref{prop:levyweakorder}, one gets
      \begin{align*}
			\big|\E\big[\Phi(Y(E_{\Delta}(T))
			- \Phi(Y_{\Delta}(E_{\Delta}(T))\big]\big|
			\leq
			\Delta C\E\big[e^{CE_{\Delta}(T)}\big]
			\leq
			\Delta C\E\big[e^{CE(T)}\big]
			\leq 
			\Delta Ce^{CT}.
	  \end{align*}
	  As a consequence, we obtain the desired result, and complete the proof.
\end{proof}

\section{Numerical experiments}\label{sect:numexp}			
This section presents two concrete examples to illustrate the theoretical results established above. As the general principles and implementation details for numerical experiments on stochastic differential equations driven by time-changed Lévy noise have already been presented in detail in \cite[Section 5]{chen2026strong}, we omit these standard ingredients here and focus instead on the numerical verification of the weak convergence order established in this paper. In all the experiments below, the nonlinear equations generated by the implicit discretization at each time step are solved by Newton–Raphson iterations with tolerance $10^{-5}$, while the expectations are approximated by Monte Carlo simulations with $20000$ sample paths.
\begin{example}\label{ex:onedimension}
\upshape 
      Consider the Ornstein--Uhlenbeck process driven by time-changed L\'{e}vy noise
      \begin{align}\label{eq:onedimensional}
            \diff{X(t)}
            =
            a_{1}(a_{2}-X(t-))\diff{E(t)}
            +
            a_{3}\diff{W(E(t))}
            +
            \int_{|z| < 1} a_{4} z
            \,\widetilde{N}(\dif{z},\dif{E(t)}),
            \quad t\in(0,T]      
      \end{align}
      with $X(0) = 0.5$, where $\{W(t)\}_{t \geq 0}$ is a $\R$-valued standard Brownian motion on the complete filtered probability space $(\Omega,\F,\P,\{\F_{t}\}_{t \geq 0})$. Let $\widetilde{N}(\dif{z},\dif{t}) := N(\dif{z},\dif{t}) - \mu(\dif{z})\dif{t}$ be the compensator of $\{\F_{t}\}_{t \geq 0}$-adapted Poisson random measure $N(\dif{z},\dif{t})$, where $\mu$ is a L\'{e}vy measure with $\mu(\dif{z}) = \frac{3}{2\sqrt{2\pi}} e^{-\frac{z^{2}}{2}}\dif{z}, z \in (-1,1)$ \cite{wang2016numerical}. Also, let $\{D(t)\}_{t \geq 0}$ be a $\alpha$-stable subordinator with $\alpha = 0.8 \in (0,1)$ and let $E(t) := \inf\big\{s \geq 0 : D(s) > t \big\}, t \geq 0$ be its inverse. We always assume that $\{W(t)\}_{t \geq 0}$, $\{N(t,\cdot)\}_{t \geq 0}$ and $\{D(t)\}_{t \geq 0}$ are mutually independent. In the numerical experiments, we choose $a_{1} = 2$, $a_{2} = 1$, $a_{3} = 0.6$ and $a_{4} = 0.5$. Under this choice of parameters, equation \eqref{eq:onedimensional} satisfies Assumptions \ref{ass:globallinear}, \ref{ass:techniquecond}, \ref{As:fghgrowthcond}, and \ref{ass:fghtimeLipschitz}, as well as the required regularity condition. Therefore, Proposition \ref{prop:levyweakorder} and Theorem \ref{th:weakorder} are applicable. 
   
      \begin{figure}[!htbp]
      \begin{center}
            \subfigure[$D(t_{n})$]
	          {\includegraphics[width=0.30\textwidth]
	          {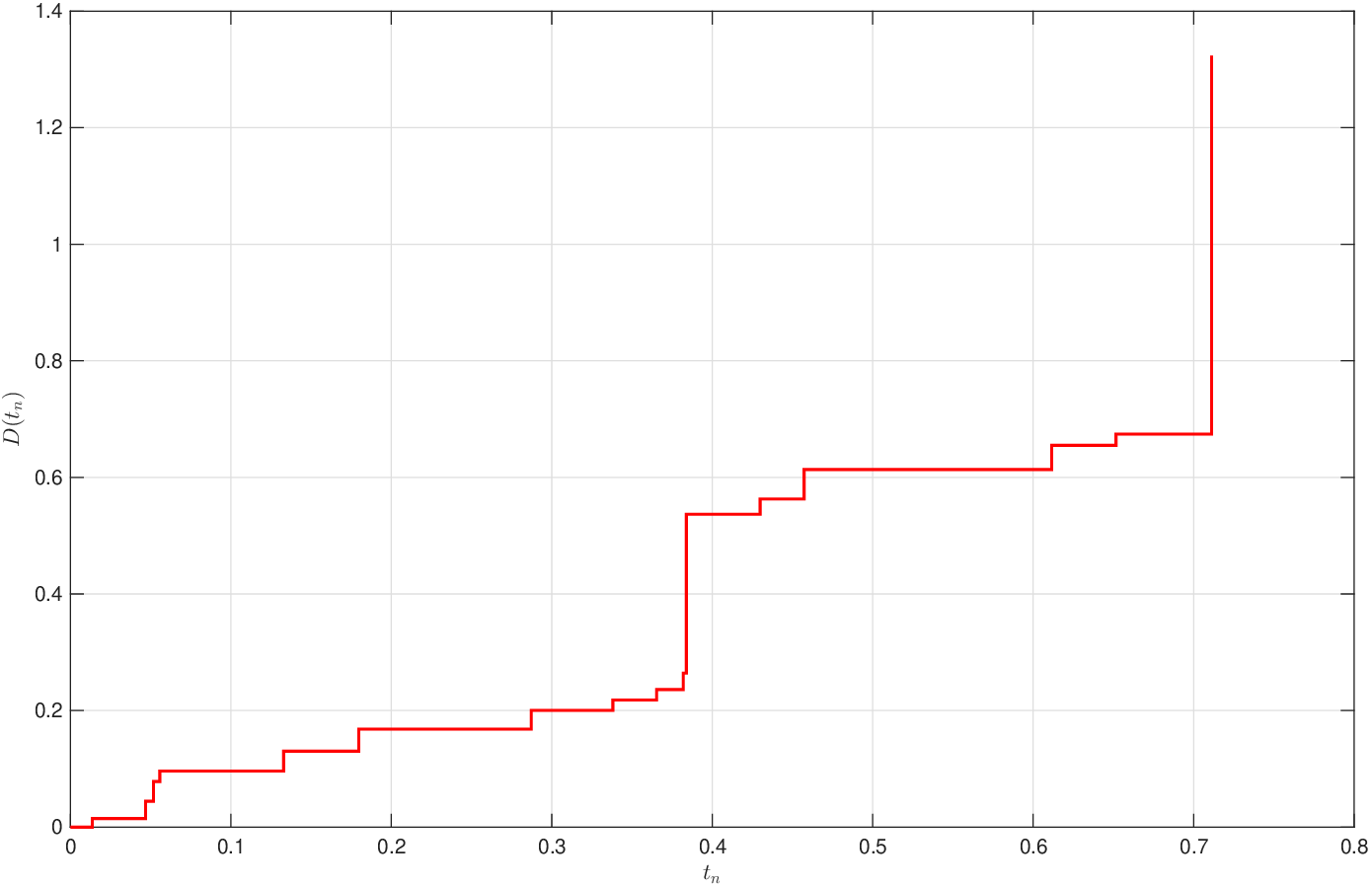}}
            \subfigure[$E_{\Delta}(t)$]
            {\includegraphics[width=0.30\textwidth]
            {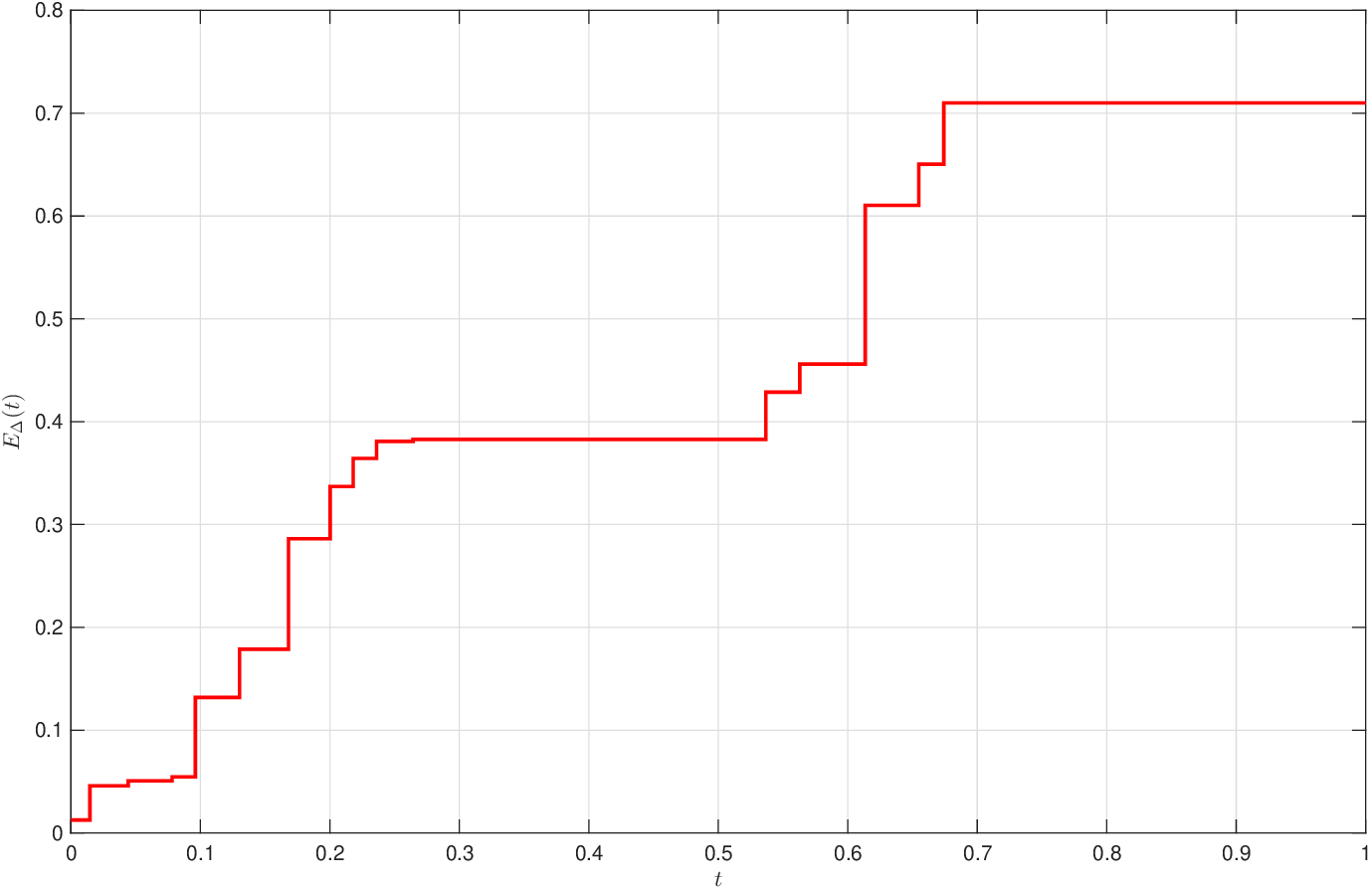}}
            \subfigure[$X_{\Delta}(t)$]
            {\includegraphics[width=0.30\textwidth]
            {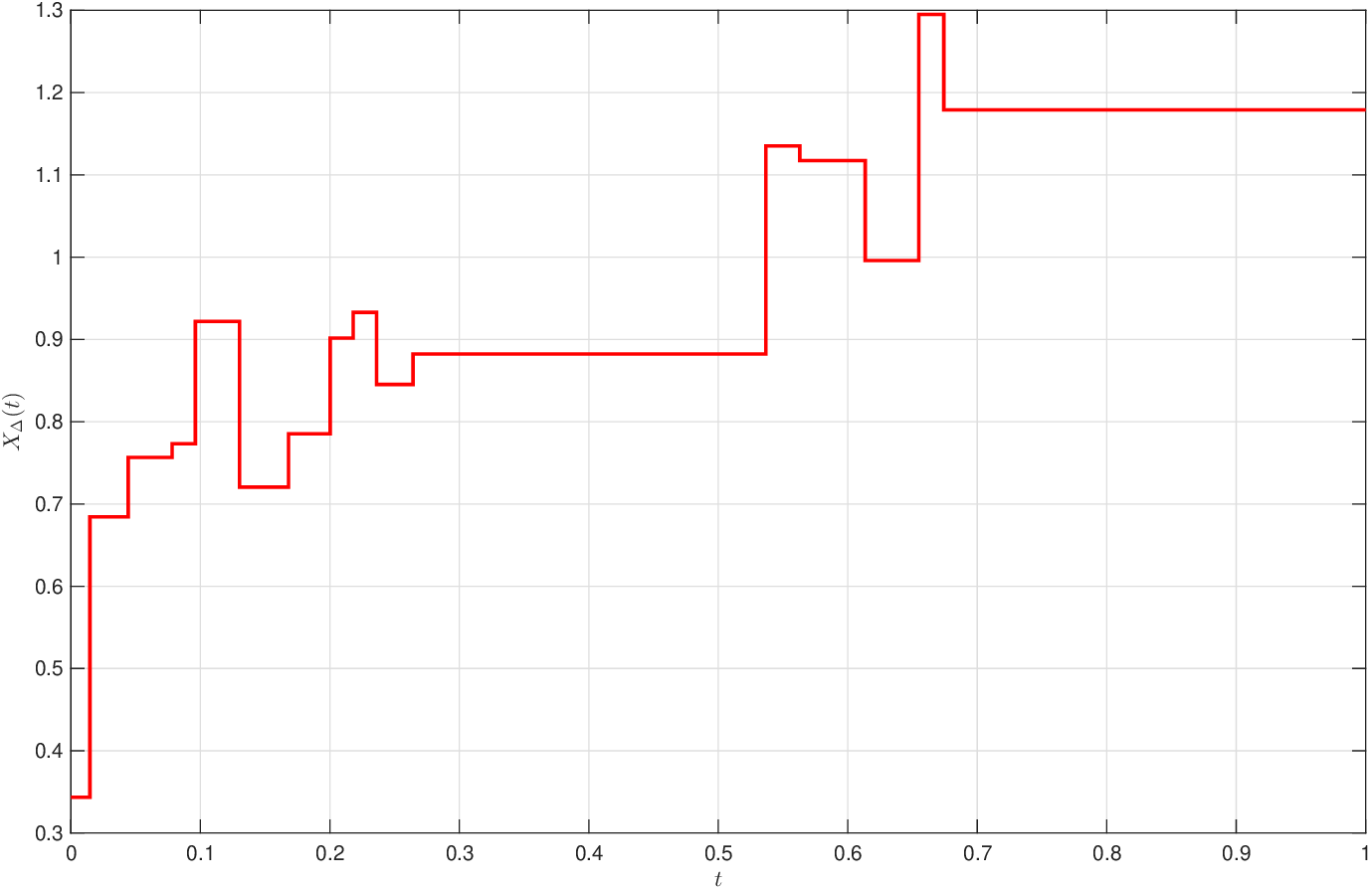}}
            \centering\caption {Sample paths of stochastic 
            processes with $\Delta = 2^{-10}$ and $\theta = 0.5$ 
            for \eqref{eq:onedimensional}}
            \label{fig:samplepaths}
      \end{center}
      \end{figure}

      For a given stepsize $\Delta \in (0,1)$, let $t_{n} := n\Delta$ for $n \in \N$. Following the LePage series representation, we generate the discrete values $\{D(t_{n})\}_{n \in \N}$ of the subordinator $\{D(t)\}_{t \geq 0}$; see \cite[Section 3.3]{janicki1994simulation} and \cite[Section 5]{chen2026strong} for details. Throughout the numerical experiments, the truncation level in the LePage series is fixed as $K = 1000$. A sample path of $D(t)$, generated with $\Delta = 2^{-10}$, is shown in Figure~\ref{fig:samplepaths}(a). Given $T = 1$, we follow the idea in \cite{jum2016strong} and use the discrete values $\{D(t_{n})\}_{n \in \N}$ to determine $N \in \N$ such that $T \in [D(t_{N}),D(t_{N+1})$. Owing to \eqref{eq:EDeltattn}, the inverse subordinator $E_{\Delta}(t)$ is a step function on $[0,T]$, and a sample path of $E_{\Delta}(t)$ is displayed in Figure~\ref{fig:samplepaths}(b). Let $\{Y_{n}\}_{n = 0}^{N}$ be generated by applying the stochastic theta method to the corresponding non-time-changed SDEs of \eqref{eq:onedimensional}. By \eqref{eq:XDeltaYDelta} and \eqref{eq:YDeltaYn}, we have $X_{\Delta}(D(t_{n})) = Y_{n}$ for $n = 0,1,\cdots, N$ and $X_{\Delta}(T) = Y_{N}$. A sample path of the numerical solution $X_{\Delta}(t)$ is plotted in Figure~\ref{fig:samplepaths}(c).
    
      \begin{figure}[!htbp]
      \begin{center}
            \subfigure[$\theta = 0$]
            {\includegraphics[width=0.30\textwidth]
            {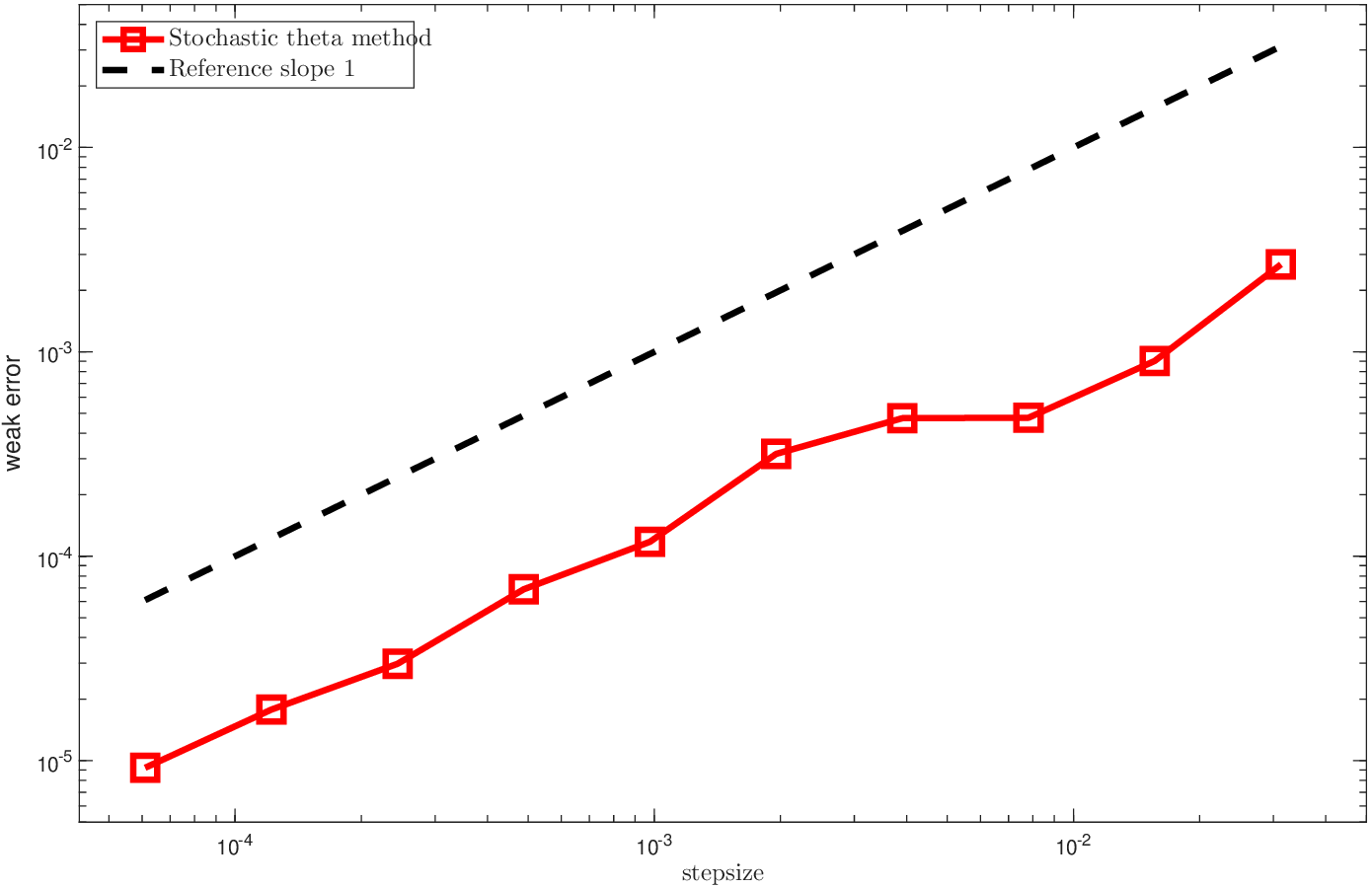}}
            \subfigure[$\theta = 0.5$]
            {\includegraphics[width=0.30\textwidth]
            {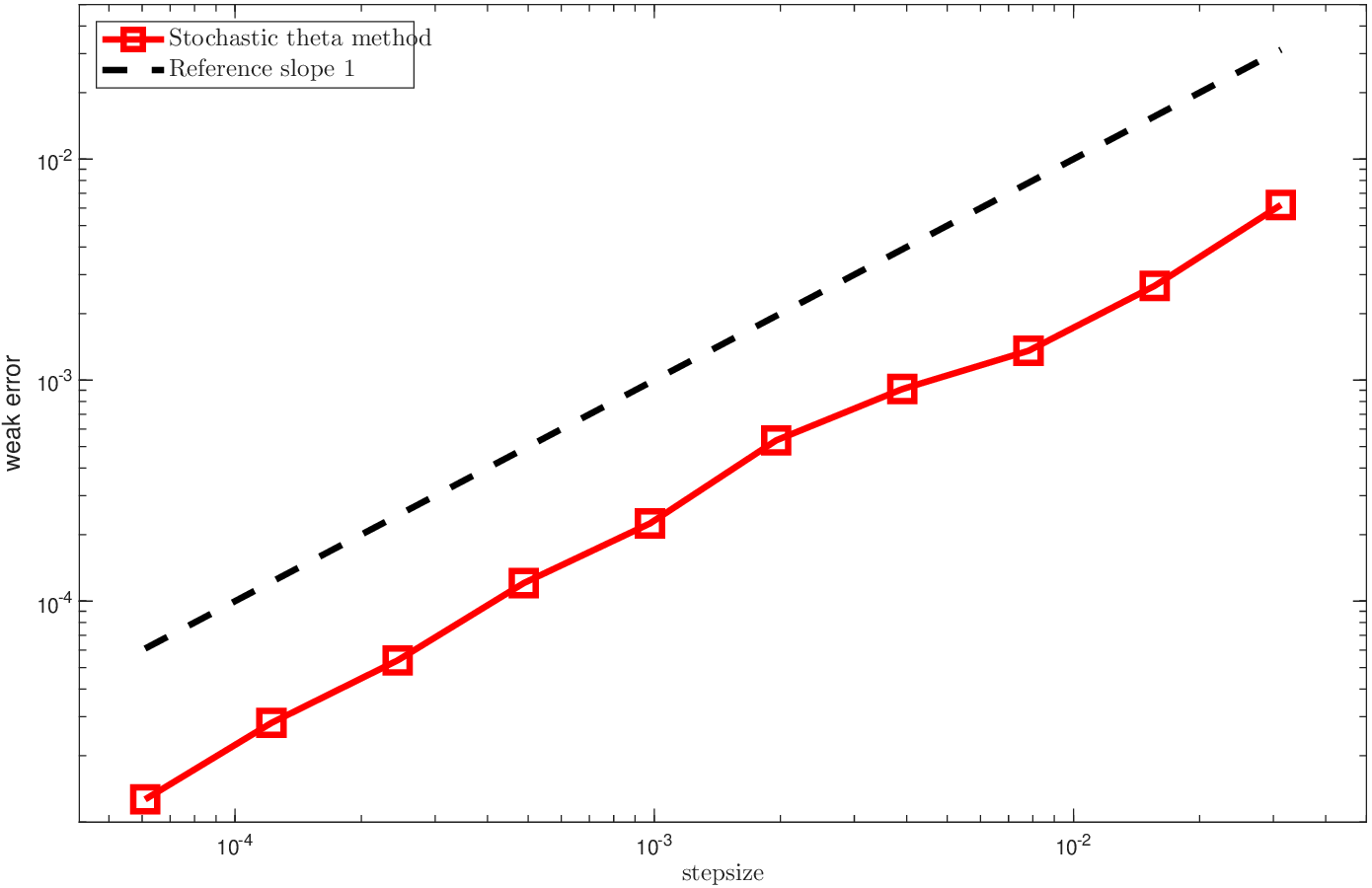}}
            \subfigure[$\theta = 1$]
            {\includegraphics[width=0.30\textwidth]
            {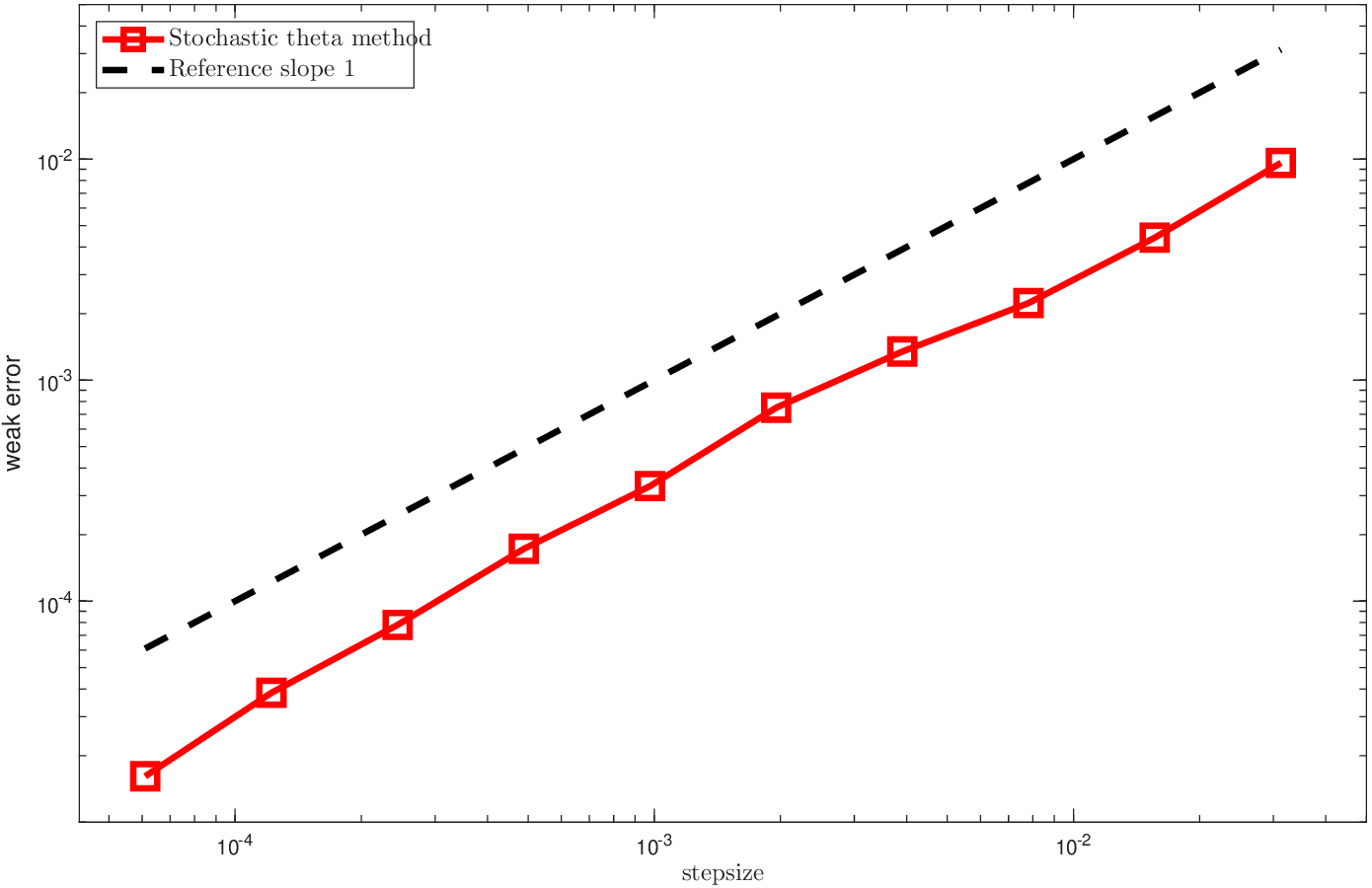}}
            \caption{Weak convergence 
            order with test function $\Phi(x) = e^{-x^{2}}$ 
            for \eqref{eq:onedimensional}}
            \label{fig:weakorderone}
      \end{center}
      \end{figure}

      To numerically verify the weak convergence order for \eqref{eq:onedimensional}, we identify the unavailable exact solution with a fine-step numerical solution with  $\Delta = 2^{-16}$. The numerical approximations are then computed by ten stepsizes $\Delta = 2^{-15}, 2^{-14}, \cdots, 2^{-6}$. Figure \ref{fig:weakorderone} shows that for different values of $\theta$, the weak error curves are parallel to the reference lines, confirming the weak order one for the stochatic theta method.
\end{example}

\begin{example}
\upshape 
      Consider the Kubo oscillator driven by time-changed L\'{e}vy noise
      \begin{align}\label{eq:twodimensional}
            \diff{\left[\begin{array}{c} \notag
                 X_{1}(t) \\ X_{2}(t) \\
            \end{array}\right]}
            =&~
            \left[\begin{array}{cc}
                0 & -a \\
                a & 0 \\
            \end{array}\right]
            \left[\begin{array}{c}
            X_{1}(t-) \\ X_{2}(t-) \\
            \end{array}\right]\diff{E(t)}
            +
            \sigma\left[\begin{array}{c}
                  X_{1}(t-) \\ X_{2}(t-) \\
            \end{array}\right]\diff{W(E(t))}
            \\&~+
            \int_{|z| < 1} \gamma 
            \left[\begin{array}{c}
            X_{2}(t-) \\ X_{1}(t-) \\
            \end{array}\right]
            z \,\widetilde{N}(\dif{z},\dif{E(t)}),
            \quad t \in (0,T]  
      \end{align}
      with $(X_{1}(0), X_{2}(0)) = (1,1)$ and $a = 2$, $\sigma = 0.5$, $\gamma = 0.5$. For this example, the Brownian motion \(W\), the Poisson random measure \(N\), and the \(\alpha\)-stable subordinator \(D(t)\), together with their numerical discretizations, are chosen exactly as in Example \ref{ex:onedimension}. Under this setting, equation \eqref{eq:twodimensional} satisfies Assumptions \ref{ass:globallinear}, \ref{ass:techniquecond}, \ref{As:fghgrowthcond}, and \ref{ass:fghtimeLipschitz}, as well as the required regularity condition, which means that Proposition \ref{prop:levyweakorder} and Theorem \ref{th:weakorder} hold. Following the same experimental framework as in Example \ref{ex:onedimension}, we use a fine-step reference solution with $\Delta = 2^{-16}$ and compute the weak errors for $\Delta = 2^{-15}, 2^{-14}, \cdots, 2^{-6}$. The corresponding weak error curves are shown in Figure \ref{fig:D2weakorderone}, from which one observes that, for different values of $\theta$, the stochastic theta method again exhibits weak order one.

      \begin{figure}[!htbp]
      \begin{center}
            \subfigure[$\theta = 0$]
            {\includegraphics[width=0.30\textwidth]
            {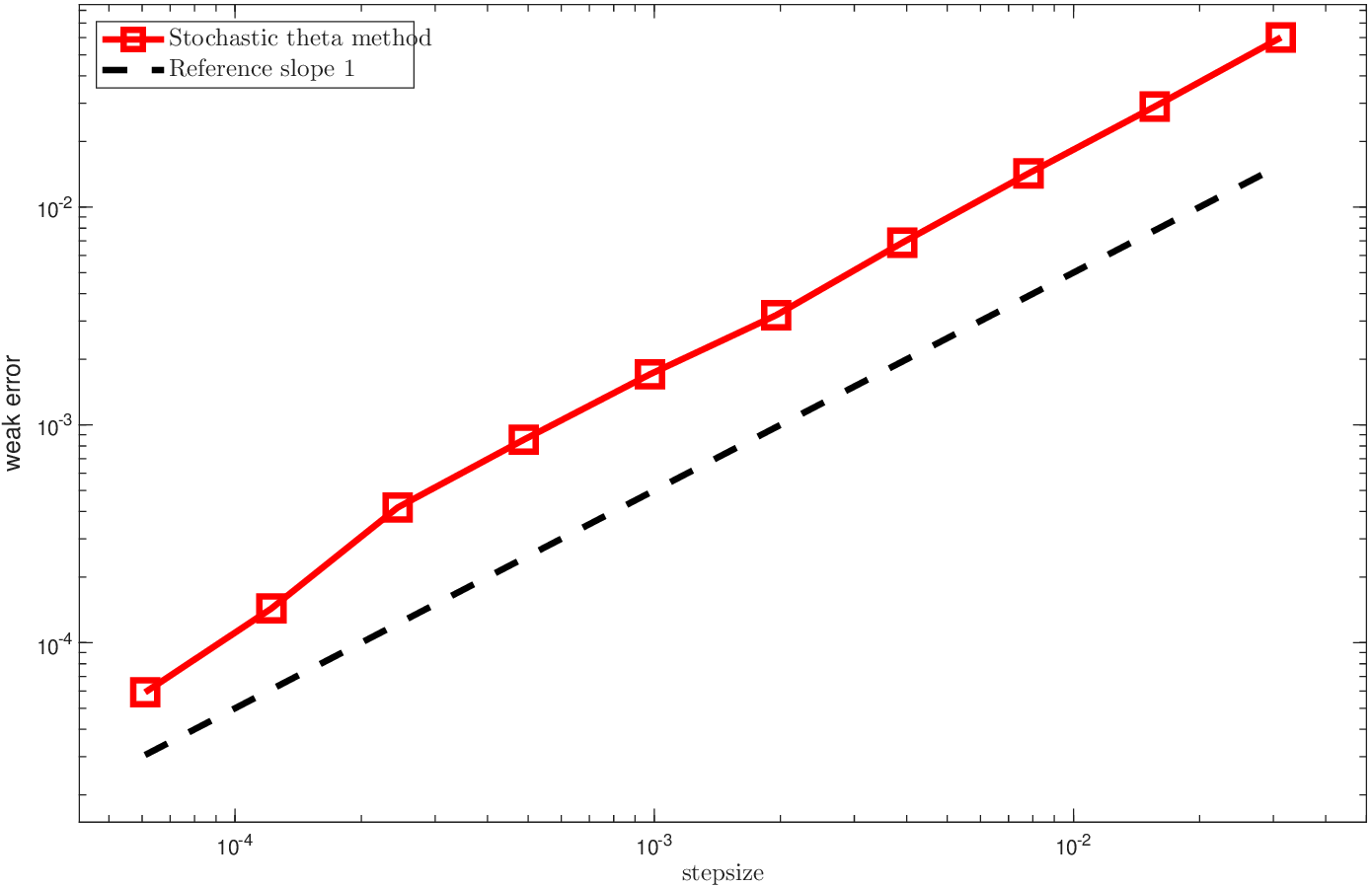}}
            \subfigure[$\theta = 0.5$]
            {\includegraphics[width=0.30\textwidth]
            {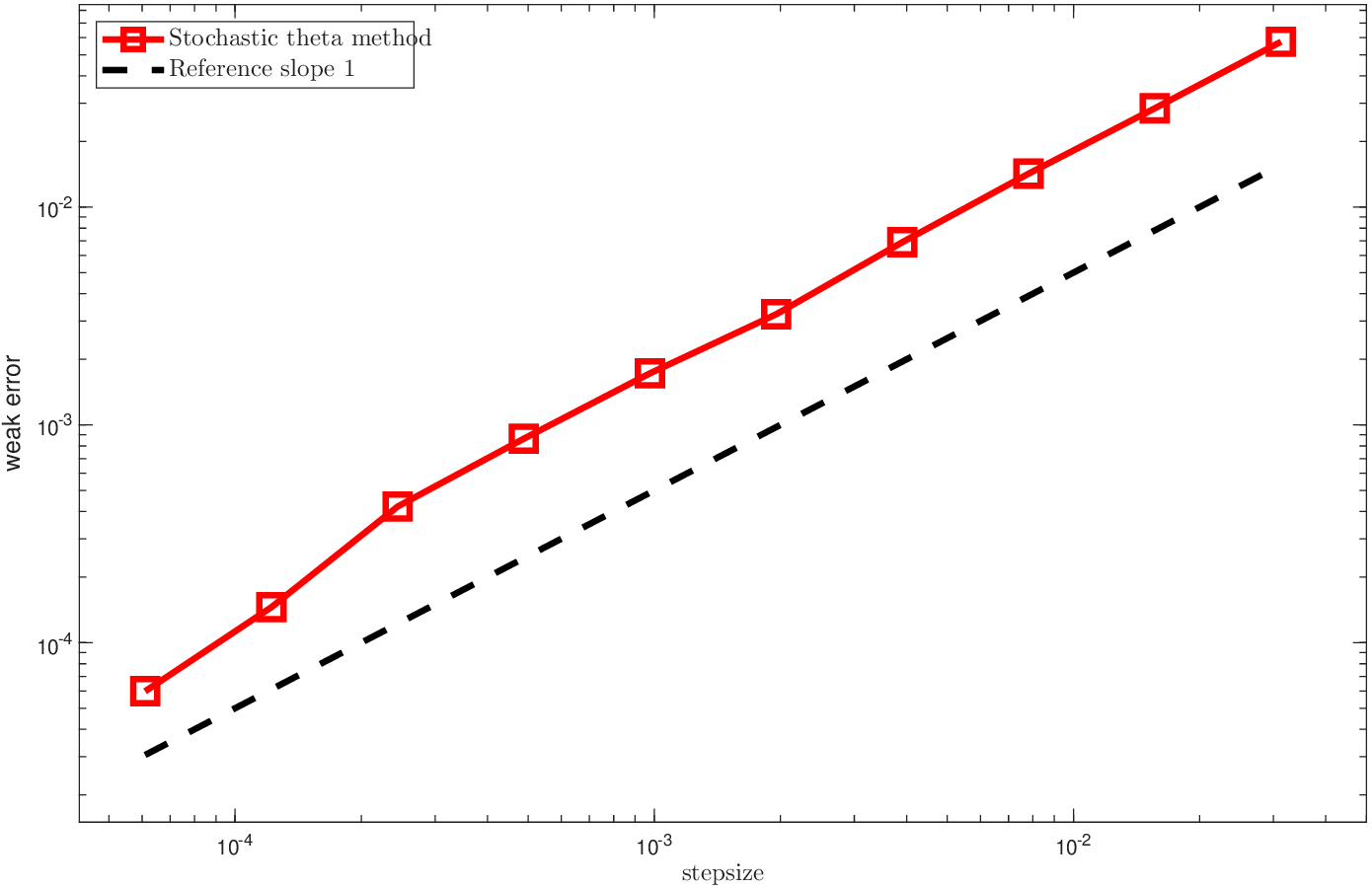}}
            \subfigure[$\theta = 1$]
            {\includegraphics[width=0.30\textwidth]
            {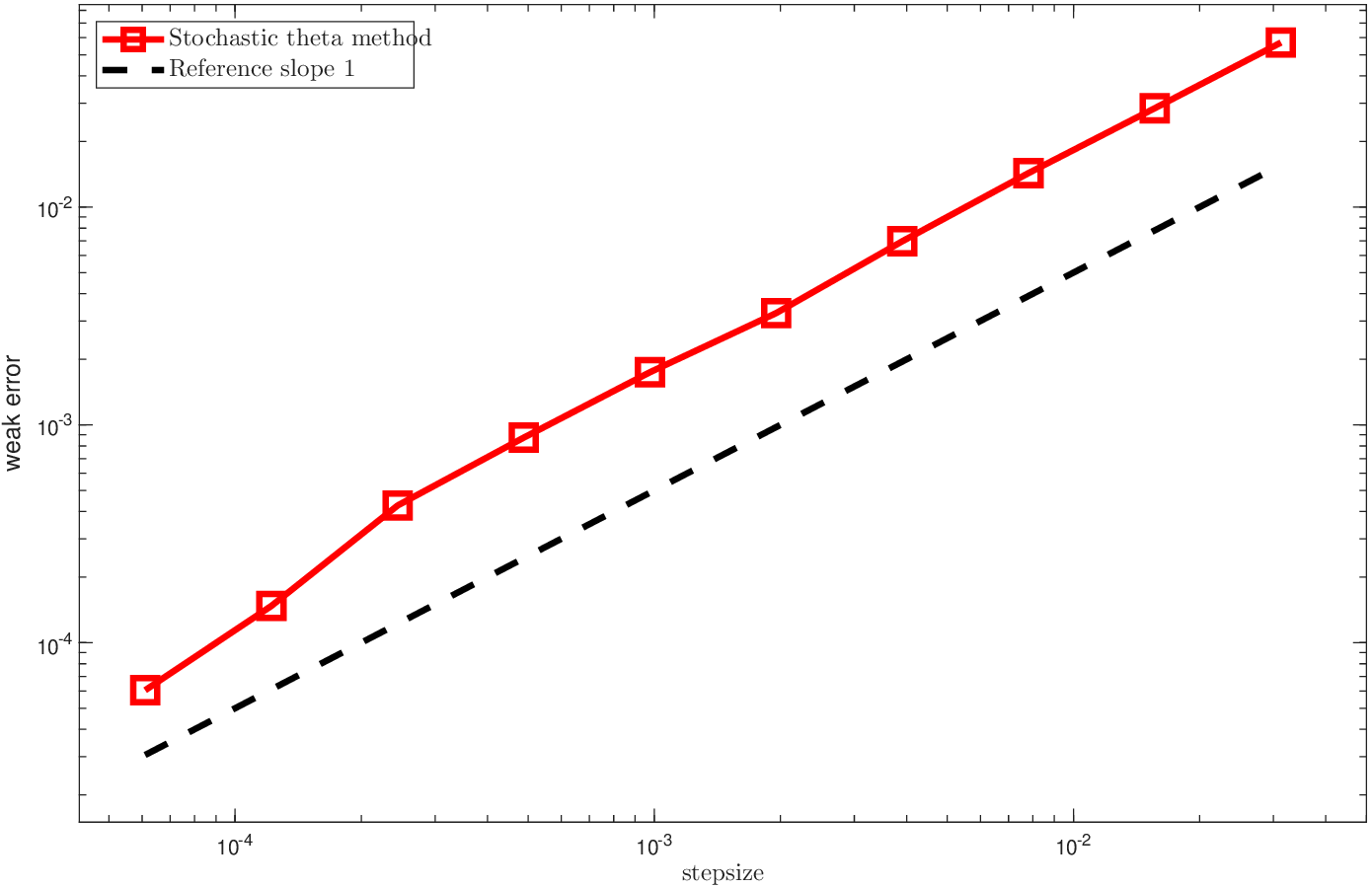}}
            \caption{Weak convergence 
            order with test function $\Phi(x_{1},x_{2}) = x_{1}x_{2}$ 
            for \eqref{eq:twodimensional}}
            \label{fig:D2weakorderone}
      \end{center}
      \end{figure}
\end{example}

\section{Conclusion}\label{sec:conclusion}		 
This work studies the weak convergence order of the stochastic theta method for SDEs driven by time-changed L\'{e}vy noise. Under global Lipschitz and linear growth conditions, we establish that the considered method attains weak order one for all $\theta \in [0,1]$. The numerical experiments further support the theoretical findings. These results extend the weak convergence analysis in the time-changed L\'{e}vy setting beyond the Euler--Maruyama method to the more general class of stochastic theta methods.
A natural direction for future research is to establish weak convergence orders of numerical methods under non-global Lipschitz conditions (see \cite{chen2019meansquare, chen2020convergence, platen2010numerical}). Other interesting topics include more general time-changed L\'{e}vy systems, such as models with infinite-activity jumps or state-dependent jump structures, as well as higher-order weak approximations in the time-changed L\'{e}vy setting.









\bibliographystyle{plain}	
\bibliography{SDETCLNBibtex}
\end{document}